\documentclass[12pt]{amsart}
\usepackage[latin9]{inputenc}
\usepackage{geometry}
\usepackage{units}
\usepackage{mathrsfs}
\usepackage{enumitem}
\usepackage{amsbsy}
\usepackage{amstext}
\usepackage{amsthm}
\usepackage{amssymb}
\usepackage{graphicx}
\usepackage{breakurl}

\makeatletter
\numberwithin{equation}{section}
\numberwithin{figure}{section}
\theoremstyle{plain}
\newtheorem{thm}{\protect\theoremname}[section]
  \theoremstyle{definition}
  \newtheorem{defn}[thm]{\protect\definitionname}
  \theoremstyle{plain}
  \newtheorem{lem}[thm]{\protect\lemmaname}
  \theoremstyle{plain}
  \newtheorem{prop}[thm]{\protect\propositionname}
  \theoremstyle{remark}
  \newtheorem*{rem*}{\protect\remarkname}
  \theoremstyle{plain}
  \newtheorem{cor}[thm]{\protect\corollaryname}
  \theoremstyle{plain}
  \newtheorem{conjecture}[thm]{\protect\conjecturename}

\@ifundefined{date}{}{\date{}}

\usepackage{amsthm}
\usepackage{amscd}
\usepackage{comment}
\usepackage{mathrsfs}
\usepackage{setspace}
\usepackage{color}

\makeatother

  \providecommand{\conjecturename}{Conjecture}
  \providecommand{\corollaryname}{Corollary}
  \providecommand{\definitionname}{Definition}
  \providecommand{\lemmaname}{Lemma}
  \providecommand{\propositionname}{Proposition}
  \providecommand{\remarkname}{Remark}
\providecommand{\theoremname}{Theorem}

\begin{document}
\global\long\def\at{\left.\right|_{\Omega}}

\global\long\def\disgraph{\mathcal{G}}

\global\long\def\metgraph{\Gamma}

\global\long\def\ui{\mathbf{\textrm{i}}}

\global\long\def\ue{\mathbf{\textrm{e}}}

\global\long\def\ud{\mathbf{\textrm{d}}}

\global\long\def\d{\partial}

\global\long\def\E{\mathcal{E}}

\global\long\def\V{\mathcal{V}}

\global\long\def\Vint{\mathcal{V}_{\mathrm{int}}}

\global\long\def\stardom{\Omega_{a,b}^{\mathrm{star}}}

\global\long\def\lensdom{\Omega_{a,b}^{\mathrm{lens}}}

\global\long\def\lapstar{\Delta_{a,b}}

\global\long\def\la{\lambda}

\global\long\def\Cr{\mathscr{C}\left(f\right)}

\global\long\def\Min{\mathscr{M}_{-}\left(f\right)}

\global\long\def\Max{\mathscr{M}_{+}\left(f\right)}

\global\long\def\Sd{\mathscr{S}\left(f\right)}

\global\long\def\Xt{\mathscr{X}\left(f\right)}

\global\long\def\Nd{\mathcal{N}\left(f\right)}

\global\long\def\hess{\mathrm{Hess}}

\global\long\def\H{H}

\global\long\def\Z{\mathbb{Z}}

\global\long\def\R{\mathbb{R}}

\global\long\def\C{\mathbb{C}}

\global\long\def\N{\mathbb{N}}

\global\long\def\Q{\mathbb{Q}}

\global\long\def\msing{\Sigma^{\mathrm{sing}}}

\global\long\def\mreg{\Sigma^{\mathrm{reg}}}

\global\long\def\mgen{\Sigma^{\mathrm{gen}}}

\global\long\def\mg{\Sigma^{\mathrm{g}}}

\global\long\def\lap{\Delta}

\global\long\def\na{\nabla}

\global\long\def\bs#1{\boldsymbol{#1}}

\global\long\def\deg#1{\mathrm{deg}(#1)}

\global\long\def\neig{\nu\left(\phi_{\eig}\right)}

\global\long\def\eig{\lambda}

\global\long\def\T{\mathbb{T}}

\global\long\def\torus{\mathbb{T^{\left|\E\right|}}}

\global\long\def\meas{\mu_{\vec{L}}}

\global\long\def\dens{d_{\vec{L}}}

\global\long\def\sgn{\mathrm{sgn}}
\global\long\def\undercom#1#2{\underset{_{#2}}{\underbrace{#1}}}

\title{Neumann Domains on Graphs and Manifolds}

\author{Lior Alon, Ram Band, Michael Bersudsky, Sebastian Egger}

\address{{\small{}Department of Mathematics, Technion--Israel Institute of
Technology, Haifa 32000, Israel}}

\begin{abstract}
The nodal set of a Laplacian eigenfunction forms a partition of the
underlying manifold or graph. Another natural partition is based on
the gradient vector field of the eigenfunction (on a manifold) or
on the extremal points of the eigenfunction (on a graph). The submanifolds
(or subgraphs) of this partition are called Neumann domains. This
paper reviews the subject, as appears in \cite{AloBan_Neumann,BanEggTay_arXiv17,BanFaj_ahp16,McDFul_ptrs13,Zel_sdg13}
and points out some open questions and conjectures. The paper concerns
both manifolds and metric graphs and the exposition allows for a comparison
between the results obtained for each of them.
\end{abstract}

\maketitle

\section{Introduction}

Given a Laplacian eigenfunction on a manifold or a metric graph, there
is a natural partition of the manifold or the graph. The partition
is dictated by the gradient vector field of the eigenfunction (on
a manifold) or by the extremal points of the eigenfunction (on a graph).
The submanifolds (or subgraphs) of such a partition are called Neumann
domains and the separating lines (or points in the case of a graph)
are called Neumann lines (or points). The counterpart of this partition
is the nodal partition (with the same terminology of nodal domains,
nodal lines and nodal points). This latter partition is extensively
studied in the last two decades or so (though interesting results
on nodal domains appeared throughout all of the 20-th century and
even earlier). When restricting an eigenfunction to a single nodal
domain one gets an eigenfunction of that domain with Dirichlet boundary
conditions. Similarly, when restricting an eigenfunction to a Neumann
domain, one gets a Neumann eigenfunction of that domain (Lemmata \ref{lem:Restriction-is-Neumann-eigenfunction-manifold},\ref{lem:Restriction-is-Neumann-eigenfunction-graph}),
which explains the name \emph{Neumann} domain and shows the most basic
linkage between nodal domains and Neumann domains.

Neumann domains form a very new topic of study in spectral geometry.
They were first mentioned in a paragraph of a manuscript by Zelditch
\cite{Zel_sdg13}. Shortly afterwards (and independently) a paper
by McDonald and Fulling was dedicated to Neumann domains \cite{McDFul_ptrs13}.
Since then two additional papers contributed to this topic; one of
the authors with Fajman \cite{BanFaj_ahp16} and two of the authors
with Taylor \cite{BanEggTay_arXiv17}. The first part of the current
manuscript serves as an exposition of the known results for Neumann
domains on two-dimensional manifolds, adding a few supplementary new
results and proofs. The second part focuses on Neumann domains on
metric graphs and reviews the results which appear in \cite{AloBan_Neumann}\footnote{While writing this manuscript, we became aware that there is an ongoing
research on the related topic of Neumann partitions on graphs. These
works in progress are done by Gregory Berkolaiko, James Kennedy, Pavel Kurasov,
Corentin L\'ena and Delio Mugnolo.}. We aim to point out similarities and differences between Neumann
domains on manifolds and those on graphs. For this purpose, each of
the two parts of the papers is divided to exactly the same subtopics:
definitions, topology, geometry, spectral position and count. We also
include an appendix which contains a short review of relevant results
in basic Morse theory, useful for the manifold part of the paper.
The summary of the paper provides guidelines for comparison between
the manifold results and the graph results. Such a comparison had
taught us a great deal in what concerns to the field of nodal domains
and yielded a wealth of new results both on manifolds and graphs.
As an example we only mention the topic of nodal partitions and refer
the interested reader to \cite{BanBerRazSmi_cmp12,Ber_apde13,BerKucSmi_gafa12,BerRazSmi_jpa12,BerWey_ptrsa14,BonHel_book17,CdV_apde13,HelHof_jems13,HelHofTer_aihp09}
in order to learn on the evolution of this research direction. In
addition to that, we believe that it is beneficial to compare problems
between the fields of nodal domains and Neumann domains. We point
out such similarities and differences throughout the paper.

Although new in spectral theory, Neumann domains were used in computational
geometry, where they are known as Morse-Smale complexes (see the book
\cite{zomorodian2005topology} or \cite{Biasoti_describing_shapes08}
for an extensive review). They are used as a tool to analyze sets
of measurements on certain spaces and for getting a good qualitative
and quantitative acquaintance with the measured functions \cite{ChaVegYap_jsc17,EdeHarZon_dcg03,EdeHarNatPas_scg03}.
Another field of relevance is computer graphics, where Morse-Smale
complexes of Laplacian eigenfunctions are applied for surface segmentation
\cite{DonBreGarPasHar_sig06,GyuBrePas_ieee12,Reuter_jcp10}.\vspace{2mm}

\part{Neumann domains on two-dimensional manifolds}

\vspace{4mm}

\section{Definitions}

\noindent Let $(M,\thinspace g)$ be a two-dimensional, connected,
orientable and closed Riemannian manifold. We denote by $-\Delta$
the (negative) self-adjoint Laplace-Beltrami operator. Its spectrum
is purely discrete since $M$ is compact. We order the eigenvalues
$\{\lambda_{n}\}_{n=0}^{\infty}$ increasingly, $0=\lambda_{0}<\lambda_{1}\leq\lambda_{2}\leq\ldots$,
and denote a corresponding complete system of orthonormal eigenfunctions
by $\{f_{n}\}_{n=0}^{\infty}$, so that we have
\begin{equation}
-\Delta f_{n}=\lambda_{n}f_{n}.
\end{equation}

\noindent We assume in the following that the eigenfunctions $f$
are Morse functions, i.e.\,have no degenerate critical points\footnote{These are critical points where the determinant of the Hessian vanishes. }.
We call such an $f$ a \emph{Morse-eigenfunction}. Eigenfunctions
are generically Morse, as shown in \cite{Albert_thesis72,Uhl_ajm76}.
At this point, we refer the interested reader to the appendix, where
some basic Morse theory which is relevant to the paper is presented.

\noindent In order to define Neumann domains and Neumann lines we
introduce the following construction based on the gradient vector
field, $\nabla f$. This vector field defines the following flow:
\begin{equation}
\begin{aligned} & \varphi:\mathbb{R}\times\,M\rightarrow M,\\
 & \partial_{t}\varphi(t,\,\bs x)=-\nabla f\big|_{\varphi(t,\,\bs x)},\\
 & \varphi(0,\,\bs x)=\bs x.
\end{aligned}
\label{eq:flow}
\end{equation}

\noindent The following notations are used throughout the paper. The
set of critical points of $f$ is denoted by $\Cr$; the sets of saddle
points and extrema of $f$ are denoted by $\Sd$ and $\Xt$; the sets
of minima and maxima of $f$ are denoted by $\Min$ and $\Max$, respectively.

\noindent For a critical point $\bs x\in\Cr$, we define its stable
and unstable manifolds by
\begin{equation}
\begin{aligned}W^{s}(\bs x) & =\{\bs y\in M\big|\lim_{t\rightarrow\infty}\varphi(t,\,\bs y)=\bs x\}\mbox{ and }\\
W^{u}(\bs x) & =\{\bs y\in M\big|\lim_{t\rightarrow-\infty}\varphi(t,\,\bs y)=\bs x\},
\end{aligned}
\label{eq:Stable-Unstable-Def}
\end{equation}
respectively. Intuitively, these notions may be visualized in terms
of surface topography; the stable manifold, $W^{s}(\bs x)$, may be
thought of as a dale (where falling rain droplets would flow and reach
$\bs x$) and the unstable manifold, $W^{u}(\bs x)$, as a hill (with
opposite meaning in terms of water flow). An interesting scientific
account on those appeared by Maxwell already in 1870 \cite{Maxwell_1870}.
\begin{defn}
\noindent \cite{BanFaj_ahp16}\label{def:Neumann-Domains-and-Lines}
Let $f$ be a Morse function.
\begin{enumerate}
\item Let $\bs p\in\Min,\,\,\bs q\in\Max$, such that $W^{s}\left(\bs p\right)\cap W^{u}\left(\bs q\right)\neq\emptyset$.
Each of the connected components of $W^{s}\left(\bs p\right)\cap W^{u}\left(\bs q\right)$
is called a \emph{Neumann domain} of $f$.
\item The \emph{Neumann line set} of $f$ is
\begin{equation}
\Nd:=\overline{\bigcup_{{\bs r\in\Sd}}W^{s}(\bs r)\cup W^{u}(\bs r)}.\label{eq:Neumann-line-set}
\end{equation}
\end{enumerate}
\end{defn}
Note that the definition above may be applied to any Morse function
and not necessarily to eigenfunctions. Indeed, some of the results
to follow do not depend on $f$ being an eigenfunction. Yet, the spectral
theoretic point of view is the one which motivates us to consider
the particular case of Laplacian eigenfunctions.\\
It is not hard to see from basic Morse theory that Neumann domains
are two-dimensional subsets of $M$, whereas the Neumann line set
is a union of one dimensional curves on $M$ (see appendix). Further
properties of Neumann domains and Neumann lines are described in the
next section.

Figure \ref{fig:Neumann-Domains-Torus-Basic} shows an eigenfunction
of the flat torus with its partition to Neumann domains.
\begin{figure}[h]
\centering{}\includegraphics[width=1\textwidth]{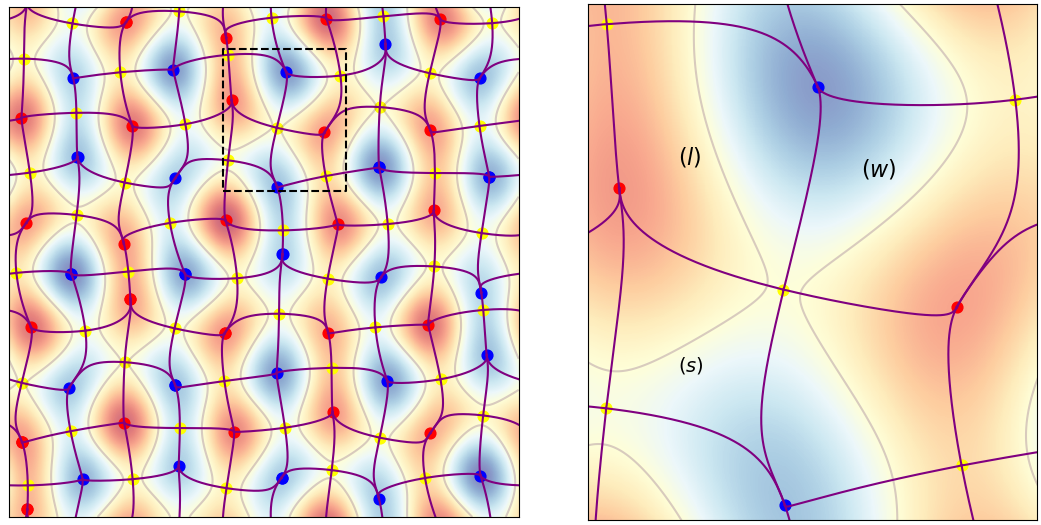} \caption{Left: An eigenfunction corresponding to eigenvalue $\lambda=25$ of
the flat torus whose fundamental domain is $[0,2\pi]\times[0,2\pi]$.
Red (blue) colors indicate positive (negative) values of the eigenfunction.
Red (blue) points mark maximum (minimum) points and yellow points
mark saddle points. The nodal set is drawn in grey and the Neumann
line set in purple. The Neumann domains are the domains bounded by
the Neumann line set.\protect \\
Right: A magnification of the marked square from the left figure.
Three Neumann domains are marked by (s), (l) and (w) according to
the three distinguished Neumann domain types described in Section
\ref{subsec:Angles}.}
\label{fig:Neumann-Domains-Torus-Basic}
\end{figure}

In the above and throughout the paper, we treat only manifolds without
boundary, in order to avoid technicalities and ease the reading. It
is possible to define Neumann domains for manifolds with boundary
and to prove analogous results for those. The interested reader is
referred to \cite{BanFaj_ahp16} for such a treatment.

\section{Topology of $\Omega$ and topography of $\left.f\right|_{\Omega}$\label{sec:Topology-and-Topography-Manifolds}}

Let $f$ be an eigenfunction corresponding to an eigenvalue $\lambda$
and let $\Omega$ be a Neumann domain. The boundary, $\partial\Omega$,
consists of Neumann lines, which are particular gradient flow lines
(see appendix). As the gradient $\nabla f$ is tangential to the Neumann
lines we get that $\left.\hat{n}\cdot\nabla f\right|_{\partial\Omega}=0$,
where $\hat{n}$ is normal to $\partial\Omega$. As a consequence
we have
\begin{lem}
\label{lem:Restriction-is-Neumann-eigenfunction-manifold}$\left.f\right|_{\Omega}$
is a Neumann eigenfunction of $\Omega$ and corresponds to the eigenvalue
$\lambda$.
\end{lem}
This lemma is the reason for the name \emph{Neumann} domains.\\

Next, we describe the topological properties of a Neumann domain $\Omega$,
as well as the topography of $\left.f\right|_{\Omega}$. By topography
of a function, we mean the information on its level sets and critical
points.
\begin{thm}
\label{thm:topological-properties-manifolds}\cite[Theorem 1.4]{BanFaj_ahp16}

Let $f$ be a Morse function with a non-empty set of saddle points,
$\Sd\neq\emptyset$.

Let $\bs p\in\Min,\,\bs q\in\Max$ with $W^{s}\left(\bs p\right)\cap W^{u}\left(\bs q\right)\neq\emptyset$.

Let $\Omega$ be a connected component of $W^{s}\left(\bs p\right)\cap W^{u}\left(\bs q\right)$,
i.e., $\Omega$ is a Neumann domain.

The following properties hold.
\begin{enumerate}
\item The Neumann domain $\Omega$ is a simply connected open set.\label{enu:thm-topological-properties-manifolds-1}
\item All critical points of $f$ belong to the Neumann line set, i.e.,
$\Cr\subset\Nd$.\label{enu:thm-topological-properties-manifolds-2}
\item The extremal points which belong to $\overline{\Omega}$ are exactly
$\bs p,\bs q$, i.e., $\Xt\cap\partial\Omega=\{\bs p,\thinspace\bs q\}$.\label{enu:thm-topological-properties-manifolds-3}
\item If $f$ is a Morse-Smale function\footnote{See appendix for the definition of a Morse-Smale function.}
then $\partial\Omega$ consists of Neumann lines connecting saddle
points with $\bs p$ or $\bs q$. In particular, $\partial\Omega$
contains either one or two saddle points (see also Proposition \ref{prop:Morse-Smale-on-2d}).
\label{enu:thm-topological-properties-manifolds-4}
\item Let $c\in\R$. such that $f(\bs p)<c<f(\bs q)$. $\overline{\Omega}\cap f^{-1}\left(c\right)$
is a smooth, non-self intersecting one-dimensional curve in $\overline{\Omega}$,
with boundary points lying on $\partial\Omega$.\label{enu:thm-topological-properties-manifolds-5}
\end{enumerate}
\end{thm}
This last theorem contains different properties of Neumann domains:
claim (\ref{enu:thm-topological-properties-manifolds-1}) concerns
the topology, claims (\ref{enu:thm-topological-properties-manifolds-2}),(\ref{enu:thm-topological-properties-manifolds-3}),(\ref{enu:thm-topological-properties-manifolds-4})
the critical points, and claim (\ref{enu:thm-topological-properties-manifolds-5})
the level sets. A special emphasize should be made for the case when
$f$ is a Morse function which is also an eigenfunction. For Laplacian
eigenfunctions we have that maxima are positive and minima are negative,
i.e., $f(\bs p)<0,~~f(\bs q)>0$, in the notation of the theorem.
Hence we may choose $c=0$ in claim (\ref{enu:thm-topological-properties-manifolds-5})
above and obtain a characterization of the nodal set which is contained
within a Neumann domain.

\begin{figure}[h]
\centering{}\includegraphics[width=0.4\textwidth]{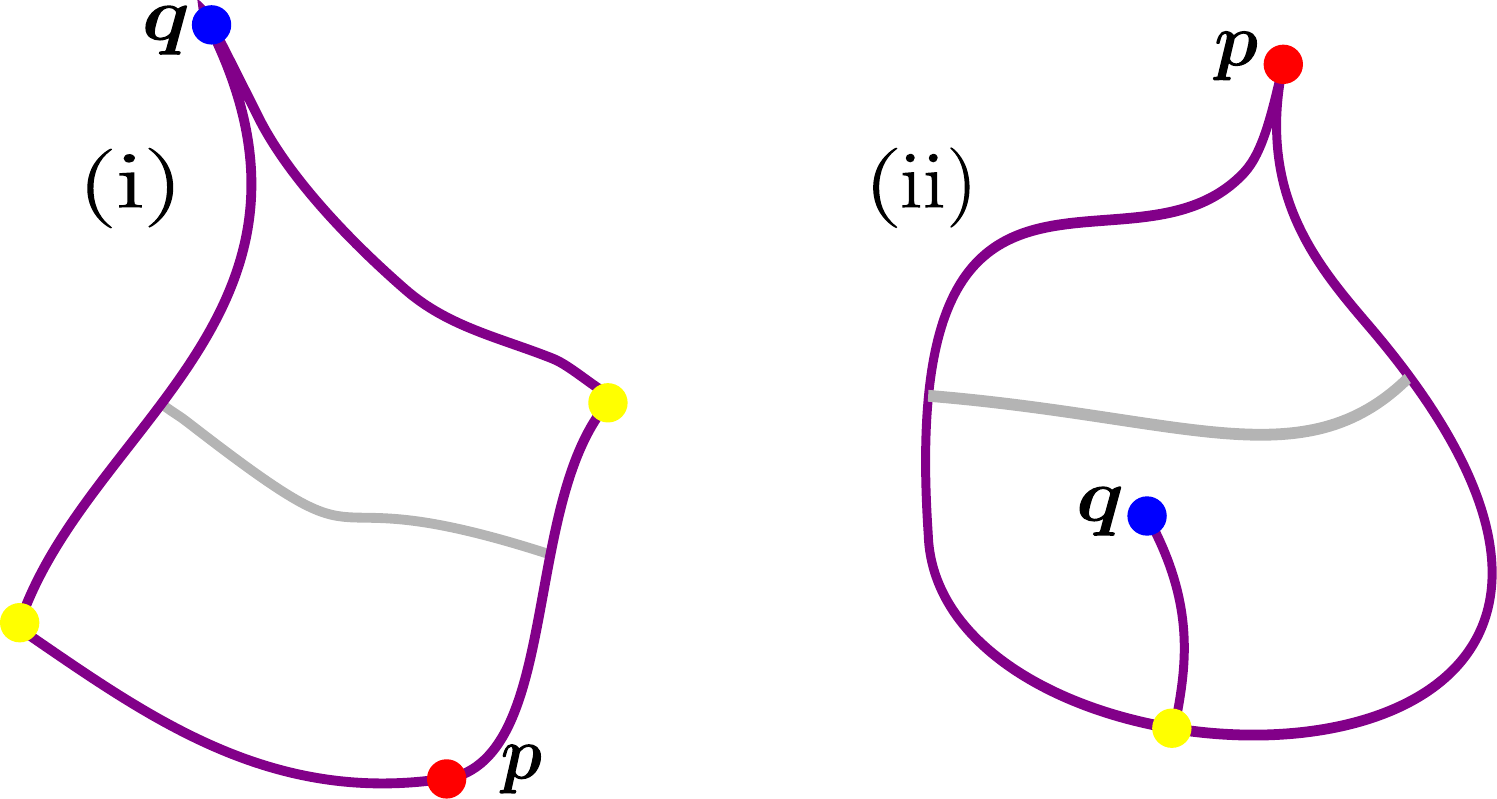} \caption{Two possible types of Neumann domains for a Morse-Smale eigenfunction.
Red (blue) discs mark maximum (minimum) points and yellow discs mark
saddle points. The nodal set is drawn in grey.}
\label{fig:Neumann-domains-schematic}
\end{figure}

Figure \ref{fig:Neumann-domains-schematic} shows the two possible
schematic shapes of Neumann domains of a Morse-Smale eigenfunction,
as implied from the properties above. We complement the figure by
noting that there exist Morse functions with Neumann domains of type
(ii) but numerical explorations have not revealed any eigenfunction
with a Neumann domain of this type.

Let us compare the results above with similar properties of nodal
domains. Nodal domains are not necessarily simply connected. On the
contrary, it was recently found that random eigenfunctions may have
nodal domains of arbitrarily high genus \cite{SarWig_cm16}. Also,
there in no upper bound on the number of critical points in a nodal
domain. A particular nodal domain may have either minima or maxima
(but not both) in its interior and saddle points both in its interior
or at its boundary.

\section{Geometry of $\Omega$}

\subsection{\label{subsec:Angles}Angles}

The angles between Neumann lines meeting at critical points are discussed
in \cite{McDFul_ptrs13}. The first two parts of the next proposition
summarize the content of theorems 3.1 and 3.2 in \cite{McDFul_ptrs13}
and further generalize their result from the Euclidean case to an
arbitrary smooth metric. The third part of the proposition is new
and concern the angles between Neumann lines and nodal lines. The
proof of the first two parts is almost the same as the one in \cite{McDFul_ptrs13}
and we bring it here for completeness.
\begin{prop}
\label{prop:angles-at-critical-pts} Let $f$ be a Morse function
on a two dimensional manifold with a smooth Riemannian metric $g$.
\begin{enumerate}
\item Let $\boldsymbol{c}$ be a saddle point of $f$. Then there are exactly
four Neumann lines meeting at $\boldsymbol{c}$ with angles $\nicefrac{\pi}{2}$.
\item Let $\boldsymbol{c}$ be an extremal point of $f$ whose Hessian is
not proportional to $g$. Then any two Neumann lines meet at $\boldsymbol{c}$
with either angle $0$, $\pi$, or $\nicefrac{\pi}{2}$.
\item Further assume that $f$ is a Morse eigenfunction.\\
Let $\bs c$ be an intersection point of a nodal line and a Neumann
line of $f$.\\
If $\bs c$ is a saddle point then the angle between those lines is
$\nicefrac{\pi}{4}$.\\
Otherwise, this angle is $\nicefrac{\pi}{2}$.
\end{enumerate}
\end{prop}
\begin{proof}
We start by some preliminaries that are relevant to proving all parts
of the proposition. Let $\boldsymbol{c}$ be an arbitrary critical
point of $f$. We may find a local coordinate system $(x,y)$ around
$\bs c$, such that $\boldsymbol{c}=\left(0,0\right)$ and $\partial_{x}$,
$\partial_{y}$ is an orthonormal basis for the tangent space $T_{\boldsymbol{c}}M$
with respect to the metric $g$ at $\bs c$. This means, in particular,
that in those coordinates, $g$ at $\bs c$ is the identity. Thus,
we get that the angle between any two vectors, $u,v\in T_{\boldsymbol{c}}M$
is given by the usual Euclidean inner product, $\left\langle u,v\right\rangle _{\R^{2}}$.\\
Next, we analyze the Neumann lines which start or end at $\bs c$.
To do that, we keep in mind that Neumann lines are gradient flow lines
which start or end at a saddle point (see appendix), so we first seek
for gradient flow lines. Using \cite[Lemma~4.4]{BanHur_MorseHomology04}
we deduce that the first (matrix-valued) coefficient in the Taylor
expansion of $\nabla f$ is $\hess f|_{\boldsymbol{c}}$. Hence, the
gradient flow equations, (\ref{eq:flow}), written in this local coordinate
system, satisfy
\begin{equation}
\begin{pmatrix}x'\left(t\right)\\
y'\left(t\right)
\end{pmatrix}=-\hess f|_{\boldsymbol{c}}\cdot\begin{pmatrix}x\left(t\right)\\
y\left(t\right)
\end{pmatrix}+O\left(\left\Vert \left(x\left(t\right),y\left(t\right)\right)\right\Vert _{\R^{2}}^{2}\right).\label{eq: diffgrad}
\end{equation}
As the Hessian is symmetric, we may diagonalize it by an orthonormal
change of the coordinates and get
\[
\hess f|_{\boldsymbol{c}}=\begin{pmatrix}\alpha_{x} & 0\\
0 & \alpha_{y}
\end{pmatrix},
\]
where $\alpha_{x},\alpha_{y}$ are both non-zero since $f$ is a Morse
function. In those new coordinates, $g$ at $\bs c$ is still the
identity. Hence, the assumption in the second part of the proposition,
that the Hessian is not proportional to $g$, is equivalent to $\alpha_{x}\neq\alpha_{y}$.
In the vicinity of $\boldsymbol{c}$ the gradient flow equations,
(\ref{eq: diffgrad}), may now be approximated by
\[
\begin{pmatrix}x'\left(t\right)\\
y'\left(t\right)
\end{pmatrix}=\begin{pmatrix}-\alpha_{x}\thinspace x\left(t\right)\\
-\alpha_{y}\thinspace y\left(t\right)
\end{pmatrix},
\]
 where we abuse notation by using $(x,y)$ again to denote the new
coordinates which diagonalize the Hessian. The solutions of the above
are
\begin{equation}
\begin{pmatrix}x\left(t\right)\\
y\left(t\right)
\end{pmatrix}=\begin{pmatrix}a_{x}e^{-\alpha_{x}t}\\
a_{y}e^{-\alpha_{y}t}
\end{pmatrix},\,\,\,\textrm{with}\,\,a_{x},a_{y},t\in\R.\label{eq:flow_lines}
\end{equation}

Consider first the case of $\alpha_{x}\ne\alpha_{y}$ both positive,
i.e., $\boldsymbol{c}$ is a minimum point. In this case, all the
flow lines (\ref{eq:flow_lines}) asymptotically converge to $\bs c$
as $t\rightarrow\infty$. Recall that $\alpha_{x}\neq\alpha_{y}$
by assumption. This allows to assume without loss of generality that
$\alpha_{y}>\alpha_{x}>0$. If $a_{x}\ne0$, we get that asymptotically
as $t\rightarrow\infty$
\[
\begin{pmatrix}x\left(t\right)\\
y\left(t\right)
\end{pmatrix}=e^{-\alpha_{x}t}\begin{pmatrix}a_{x}\\
a_{y}e^{-\left(\alpha_{y}-\alpha_{x}\right)t}
\end{pmatrix}\sim e^{-\alpha_{x}t}\begin{pmatrix}a_{x}\\
0
\end{pmatrix}.
\]
Any such flow line is tangential to the $\pm\hat{x}$ direction at
$\bs c$. This gives a continuous family of gradient flow lines, some
of which are actually also Neumann lines (this depends on whether
or not there is a saddle point at their other end, $t\rightarrow-\infty$).
Hence, the possible angles between any of those Neumann lines at $\bs c$
are either $0$ or $\pi$. In addition, if $a_{x}=0$, we get a gradient
flow line which is tangential to the $\pm\hat{y}$ direction at $\bs c$.
This gradient flow line (which is not necessarily a Neumann line)
makes an angle of $\nicefrac{\pi}{2}$ with all others. This proves
the second part of the proposition if $c$ is a minimum point. The
case of a maximum is proven in exactly the same manner.

Next we prove the first part of the proposition. If $\bs c$ is a
saddle point, then $\alpha_{x},\alpha_{y}$ are of different signs.
The only gradient flow lines, (\ref{eq:flow_lines}), which start
or end at $\bs c$ are those for which either $a_{x}=0$ or $a_{y}=0$.
At $\bs c$, these lines are either tangential to $\hat{x}$ (if $a_{y}=0$)
or tangential to $\hat{y}$ (if $a_{x}=0$). These are indeed Neumann
lines, as they are connected to a saddle point ($\bs c$). There are
four such Neumann lines, corresponding to all possible sign choices
($a_{x}=0$ and $a_{y}$ is positive\textbackslash{}negative or $a_{y}=0$
and $a_{x}$ is positive\textbackslash{}negative). The angles between
any neighbouring two lines out of the four is therefore $\nicefrac{\pi}{2}$.

Finally, we prove the third part of the proposition. If $\bs c$ is
a critical point, with $\nabla f|_{\bs c}=0$, and $f\left(\boldsymbol{c}\right)=0$
then it must be a saddle point, since maxima of a Laplacian eigenfunction
are positive and minima are negative. As $f$ is a Laplace-Beltrami
eigenfunction, we get
\begin{equation}
0=-\lambda f(\bs c)=\Delta f(\bs c)=\mathrm{trace}\hess f|_{\boldsymbol{c}}.\label{eq:trace-hessian}
\end{equation}
The sum of Hessian eigenvalues is therefore zero and we may denote
those by $\pm\alpha$. Choosing a coordinate system which diagonalizes
the Hessian at $\bs c=(0,0)$, we get
\begin{align*}
f\left(x,y\right) & =\frac{1}{2}\left(\alpha x^{2}-\alpha y^{2}\right)+O\left(\left\Vert \left(x\left(t\right),y\left(t\right)\right)\right\Vert _{\R^{2}}^{3}\right).
\end{align*}
This shows that the nodal lines of $f$ at $\boldsymbol{c}$ may be
approximated by $y=\pm x$. We have already seen in the previous part
of the proof that the Neumann lines which are connected to a saddle
point, $\boldsymbol{c}$, are tangential to either the $\hat{x}$
or the $\hat{y}$ axis and this gives an angle of $\nicefrac{\pi}{4}$
between neighbouring Neumann and nodal lines.

If $\bs c$ is not a critical point then $\nabla f|_{\boldsymbol{c}}\neq0$
and we may write $\mathrm{d}f(v)=\left\langle \nabla f|_{\boldsymbol{c}},v\right\rangle _{\R^{2}}$
for every $v\in T_{\bs c}M$. By taking $v$ in the direction of the
nodal line, we get that the angle between the Neumann line and the
nodal line at $\bs c$ is $\left\langle \nabla f|_{\boldsymbol{c}},v\right\rangle _{\R^{2}}$,
as $g$ is the identity at $\bs c$. Now, since $f$ is constant along
the nodal line we have $\mathrm{d}f(v)=0$, and get that the angle
between the nodal line and the Neumann line is $\nicefrac{\pi}{2}$.
\end{proof}
\begin{rem*}
It is also stated in \cite[theorem 3.1]{McDFul_ptrs13} that an angle
of $\nicefrac{\pi}{2}$ between Neumann lines at an extremal point
is non-generic (or ``unstable special case'', citing \cite{McDFul_ptrs13}).
The proof of the first part of the proposition clarifies why it is
so.
\end{rem*}
The angles between Neumann lines may be observed in Figures \ref{fig:Neumann-Domains-Torus-Basic}
and \ref{fig:Neumann-domains-schematic}. The exact angles in Figure
\ref{fig:Neumann-Domains-Torus-Basic} are better seen when zooming
in (see right part of the figure).

Proposition \ref{prop:angles-at-critical-pts} allows to classify
Neumann domains to three distinguished types, as was suggested in
\cite{BanEggTay_arXiv17}. Each Neumann domain has one maxima and
one minima on its boundary. Assume that the Neumann domain is of type
$(i)$ as depicted in Figure \ref{fig:Neumann-domains-schematic},
i.e., it does not have an extremal point which is connected only to
a single Neumann line. We call a Neumann domain
\begin{itemize}
\item star-like if both angles at its extremal points are $0$,
\item lens-like if both angles at its extremal points are $\pi$,
\item wedge-like if one of those angles is $0$ and the other is $\pi$.
\end{itemize}
Those three types of domains are indicated in Figure \ref{fig:Neumann-Domains-Torus-Basic}(Right)
by (s), (l), (w), correspondingly.

Note that this classification requires a couple of genericity assumptions:
that the Hessian at the extremal points is not proportional to the
metric and that Neumann lines do not meet perpendicularly at an extremal
point (see remark after Proposition \ref{prop:angles-at-critical-pts}).
Indeed, our numeric explorations reveal that Neumann domains are categorized
into those three types \cite{BanEggTay_arXiv17}.

\subsection{Area to perimeter ratio}
\begin{defn}
\label{def:Area-to-Perimeter-Manifolds} \cite{EGJS07} Let $f$ be
a Morse eigenfunction corresponding to the eigenvalue $\lambda$ and
let $\Omega$ be a Neumann domain of $f$. We define the normalized
area to perimeter ratio of $\Omega$ by
\[
\rho(\Omega):=\frac{\left|\Omega\right|}{\left|\partial\Omega\right|}\sqrt{\lambda},
\]
with $\left|\Omega\right|$ being the area of $\Omega$ and $\left|\partial\Omega\right|$
the total length of its perimeter.\\

This parameter was introduced in \cite{EGJS07} in order to quantify
the geometry of nodal domains. A related quantity, $\frac{\sqrt{\left|\Omega\right|}}{\left|\partial\Omega\right|}$,
is a classical one, and it is known to be bounded from above by $\frac{1}{2\sqrt{\pi}}$
(isoperimetric inequality \cite{Federer_book69}). The value $\frac{\left|\Omega\right|}{\left|\partial\Omega\right|}$
has also an interesting geometric meaning - it is the mean chord length
of the two-dimensional shape $\Omega$. The mean chord length is defined
as follows: consider all the parallel chords in a chosen direction
and take their average length. The mean chord length is then the uniform
average over all directions of that average length\footnote{We thank John Hannay for pointing out this interesting geometrical
meaning to us.}.

There are some numerical explorations, which were performed to study
the values of $\rho$ for Neumann domains. In \cite{BanEggTay_arXiv17}
the numerics was done for random eigenfunctions on the flat torus,
where the eigenvalues are highly degenerate. More specifically, for
a particular eigenvalue, many random eigenfunctions were chosen out
of the corresponding eigenspace and the $\rho$ value was numerically
computed for all their Neumann domains. The obtained probability distribution
of $\rho$ for three different eigenvalues is shown in Figure \ref{fig:rho-distribution-manifolds},(i).
A few interesting observations can be made from those plots. First,
it seems that the probability distribution does not depend on the
eigenvalue. Furthermore, in Figure \ref{fig:rho-distribution-manifolds},(ii)
the distribution was drawn separately for each of the three types
of Neumann domains mentioned in the previous subsection (star, lens
and wedge). The lens-like domains tend to get higher $\rho$ values,
star-like domains get lower values and the wedge-like are intermediate.
Another conclusion which may be drawn from these plots is related
to the spectral position of the Neumann domains, which is described
in detail in the next section.
\end{defn}
\begin{figure}[h]
\centering{}\includegraphics[width=1\textwidth]{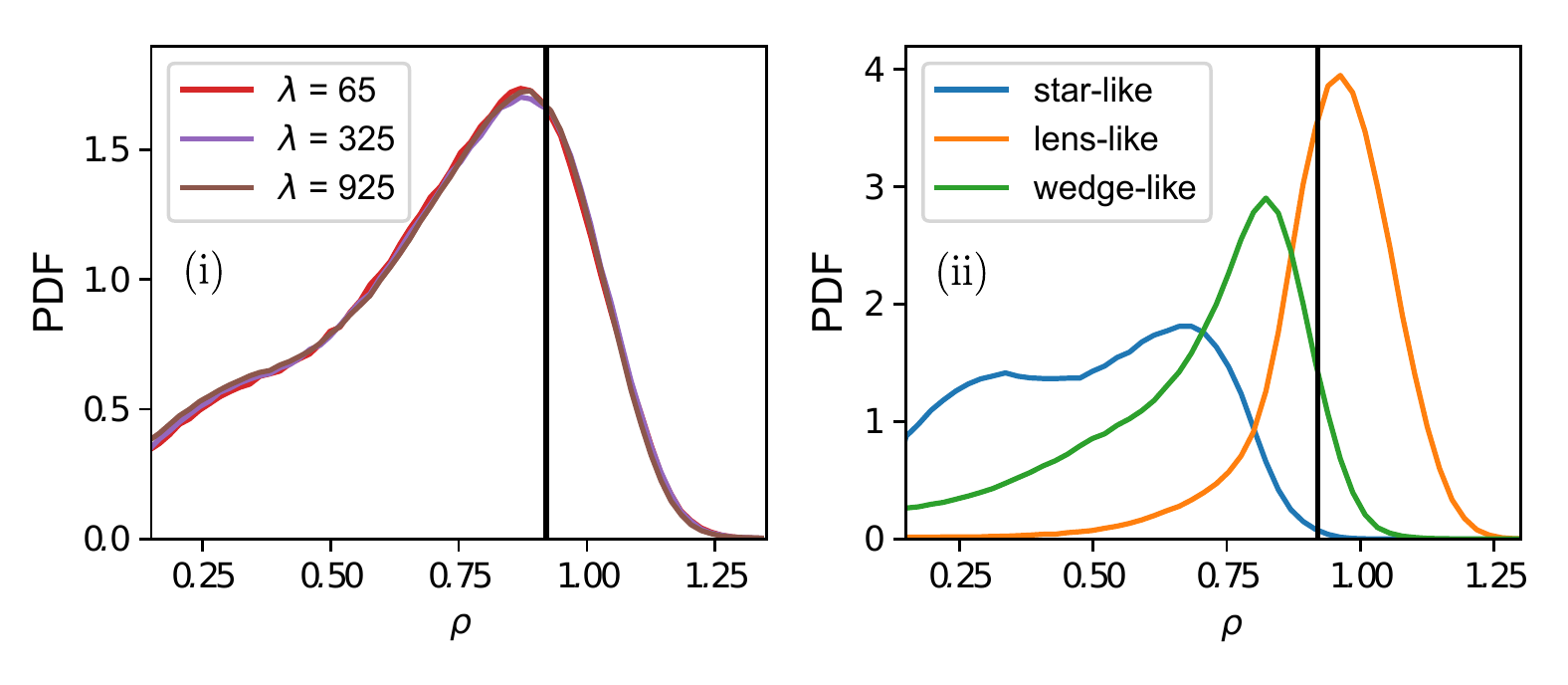}
\caption{(i): A probability distribution function of $\rho$-values of Neumann
domains for three different eigenvalues, (ii): A probability distribution
function of $\rho$-values of Neumann domains for $\lambda=925$ for
lens-like, wedge-like and star-like domains. The vertical black line
marks the value $\rho\approx0.9206$ (see Proposition \ref{prop:Rho-upper-bound-manifold},(\ref{enu:prop-Rho-upper-bound-manifold-2})).
The numerical data was calculated for approximately 9000 eigenfunctions
for each eigenvalue. The right plot is based on data of approximately
$8.5\cdot10^{6}$ Neumann domains. }
\label{fig:rho-distribution-manifolds}
\end{figure}

We may compare those results with the ones obtained for the distribution
of $\rho$ for nodal domains \cite{EGJS07}. It is shown in \cite{EGJS07}
that for nodal domains of separable eigenfunctions $\frac{\pi}{4}<\rho<\frac{\pi}{2}$.
Furthermore, it is numerically observed there that these bounds are
satisfied with probability $1$ for random eigenfunctions. Also, the
calculated probability distribution of $\rho$ for nodal domains looks
qualitatively different when comparing to Figure \ref{fig:rho-distribution-manifolds}
(see for example figures 1,2,6 in \cite{EGJS07}).

\section{Spectral position of $\Omega$\label{sec:Spectral-Position-Manifolds}}

Consider a nodal domain $\Xi$ of some eigenfunction $f$ corresponding
to an eigenvalue $\lambda$. It is known that $\left.f\right|_{\Xi}$
is the first eigenfunction (ground-state) of $\Xi$ with Dirichlet
boundary conditions \cite{Cou_ngwgmp23}. Equivalently, $\lambda$
is the lowest eigenvalue in the Dirichlet spectrum of $\Xi$. This
observation is fundamental in many results concerning nodal domains
and their counting. In this section we consider the analogous statement
for Neumann domains. Our starting point is Lemma \ref{lem:Restriction-is-Neumann-eigenfunction-manifold},
according to which an eigenvalue $\lambda$ appears in the Neumann
spectrum of each of its Neumann domains. This allows the following
definition.
\begin{defn}
\label{def:Spectral-Position} Let $f$ be a Morse eigenfunction of
an eigenvalue $\lambda$ and let $\Omega$ be a Neumann domain of
$f$. We define the spectral position of $\Omega$ as the position
of $\lambda$ in the Neumann spectrum of $\Omega$. It is explicitly
given by
\begin{equation}
N_{\Omega}(\lambda):=\left|\left\{ \lambda_{n}\in\mathrm{Spec}(\Omega)~:~\lambda_{n}<\lambda\right\} \right|,\label{eq:Spectral-Position}
\end{equation}
where $\mathrm{Spec}(\Omega):=\{\lambda_{n}\}_{n=0}^{\infty}$ is
the Neumann spectrum of $\Omega$, containing multiple appearances
of degenerate eigenvalues and including $\lambda_{0}=0$.\\
\end{defn}
\begin{rem*}
~
\begin{enumerate}
\item It can be shown (see \cite{BanEggTay_arXiv17}) that if $\Omega$
is a Neumann domain, then its Neumann spectrum is purely discrete.
This makes the above well-defined.
\item If $\lambda$ is a degenerate eigenvalue of $\Omega$, then by this
definition the spectral position is the lowest position of $\lambda$
in the spectrum.
\item For any Neumann domain, $N_{\Omega}(\lambda)>0$. Indeed, $N_{\Omega}(\lambda)=0$
is possible only for $\lambda=0$, but the zero eigenvalue corresponds
to the constant eigenfunction and this does not have Neumann domains
at all.
\end{enumerate}
\end{rem*}
A qualitative feeling on the value of $N_{\Omega}(\lambda)$ might
be given by Theorem \ref{thm:topological-properties-manifolds}. This
theorem implies that the topography of $f|_{\Omega}$ cannot be too
complex; its domain, $\Omega$, is simply connected domain; $f|_{\Omega}$
has no critical points in the interior of $\Omega$; and its zero
set is merely a single simple non-intersecting curve. These observations
suggest that $f|_{\Omega}$ might not lie too high in the spectrum
of $\Omega$. Such a belief is also apparent in \cite{Zel_sdg13},
where it is written that possibly, the spectral position of Neumann
domains 'often' equals one, just as in the case of nodal domains.
Our task is to study the possible values of $N_{\Omega}(\lambda)$
for various eigenfunctions and their Neumann domains and to investigate
to what extent $\lambda$ is indeed the first non trivial eigenvalue
of $\Omega$ ($N_{\Omega}(\lambda)=1$). We proceed by relating the
spectral position and the area to perimeter ratio (Definition \ref{def:Area-to-Perimeter-Manifolds}).

\subsection{Connecting spectral position and area to perimeter ratio}

The spectral position may be used to bound from above the area to
perimeter ratio. This holds as the area to perimeter ratio may be
written as
\[
\rho(\Omega)=\frac{\sqrt{\left|\Omega\right|}}{\left|\partial\Omega\right|}\sqrt{\left|\Omega\right|\lambda},
\]
where the first factor is bounded from above by the classical geometric
isoperimetric inequality $\frac{\sqrt{\left|\Omega\right|}}{\left|\partial\Omega\right|}\leq\frac{1}{2\sqrt{\pi}}$
\cite{Federer_book69}, and the second factor is bounded from above
by the spectral isoperimetric inequality, once the spectral position
is known. We state below the exact result, whose proof is given in
\cite{BanEggTay_arXiv17}.
\begin{prop}
\label{prop:Rho-upper-bound-manifold} \cite{BanEggTay_arXiv17} Let
$f$ be a Morse eigenfunction corresponding to eigenvalue $\lambda$.
Let $\Omega$ be a Neumann domain of $f$. We have
\begin{enumerate}
\item $\rho(\Omega)\leq\sqrt{2}N_{\Omega}(\lambda)$.\label{enu:prop-Rho-upper-bound-manifold-1}
\item if $N_{\Omega}(\lambda)=1$ then $\rho(\Omega)\leq\frac{j^{2}}{2}\approx0.9206$\label{enu:prop-Rho-upper-bound-manifold-2}
\item if $N_{\Omega}(\lambda)=2$ then $\rho(\Omega)\leq\frac{j^{2}}{\sqrt{2}}\approx1.3019$,
\label{enu:prop-Rho-upper-bound-manifold-3}
\end{enumerate}
where $j$ denotes the first zero of the derivative of the $J_{1}$
Bessel function.
\end{prop}
The bounds above may be used to gather information on the spectral
position. The calculation of $\rho(\Omega)$ is easier (either numerically
or sometimes even analytically) than this of $N_{\Omega}(\lambda)$.
As an example, we bring the probability distribution of $\rho$ given
in Figure \ref{fig:rho-distribution-manifolds},(i). The distribution
was calculated numerically for random eigenfunctions on the torus.
It is easy to observe that a substantial proportion of the Neumann
domains have a $\rho$ value which is larger than $0.9206$, the upper
bound given in Proposition \ref{prop:Rho-upper-bound-manifold}$(ii)$.
Hence, all those Neumann domains have spectral position which is larger
than one, $N_{\Omega}(\lambda)>1$. We note that those results seem
to be independent of the particular eigenvalue, as the $\rho$ distribution
itself seem not to depend on the eigenvalue. Those results are somewhat
counter-intuitive, due to what is written above (see discussion after
Definition \ref{def:Spectral-Position}). Furthermore, when calculating
the $\rho$ distribution separately for each of the three different
types of Neumann domains (Figure \ref{fig:rho-distribution-manifolds},(ii)),
the higher $\rho$ values of lens-like domains suggest that the spectral
position of those domains is higher. These results call for some further
investigation of the spectral position dependence on the shape of
the Neumann domains.

\subsection{Separable eigenfunctions on the torus}

The general problem of analytically determining the spectral position
is quite involved. Yet, there are some interesting results obtained
for separable eigenfunctions on the torus, which we review next. We
consider the flat torus with fundamental domain $\R^{2}/\Z^{2}$ equipped
with the Laplace operator. The eigenvalues are
\begin{align}
\lambda_{a,b}: & =\frac{\pi^{2}}{4}\left(\frac{1}{a^{2}}+\frac{1}{b^{2}}\right),\label{eq:Torus-eigenvalue}
\end{align}
 where
\begin{equation}
a:=\frac{1}{4m_{x}},~~b:=\frac{1}{4m_{y}},~~~\quad\textrm{for }m_{x},m_{y}\in\N.\label{eq:quantum-numbers-separable-efunc}
\end{equation}

We consider in the following only the separable eigenfunctions, which
may be written as
\begin{equation}
f_{a,b}(x,y)=\sin\left(\frac{\pi}{2a}x\right)\cos\left(\frac{\pi}{2b}y\right).\label{eq:Torus-eigenfunction}
\end{equation}
 Half of the Neumann domains of this eigenfunction are star-like and
congruent to each other and the other half are lens-like and also
congruent (Figure \ref{fig:Torus-with-star-and-lens}). We denote
those domains by $\stardom$ (Figure \ref{fig:Torus-with-star-and-lens}(ii))
and $\lensdom$ (Figure \ref{fig:Torus-with-star-and-lens}(iii)),
respectively, and in the following we investigate their spectral position.
\begin{figure}
\centering{}\includegraphics[width=1\textwidth]{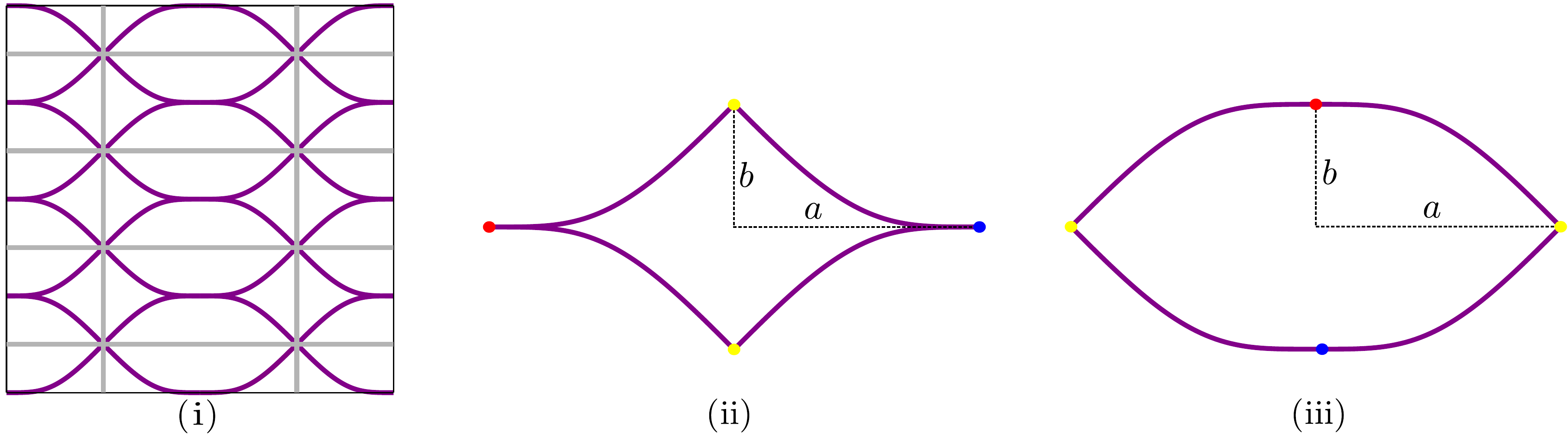} \caption{(i): Grey lines indicate the nodal set and purple lines indicate the
Neumann set of a torus eigenfunction $f(x,y)=\sin(2\pi x)\cos(4\pi y)$.
(ii) and (iii): the star-like and lens-like Neumann domains of a separable
eigenfunction (\ref{eq:Torus-eigenfunction}), with the typical lengths
$a,\thinspace b$ marked as dashed lines. Saddle points are marked
by yellow points and extrema by blue and red points.}
\label{fig:Torus-with-star-and-lens}
\end{figure}

First, we may consider only the case $b\leq a$ thanks to the symmetry
of the problem. Second, the spectral position of either $\stardom$
or $\lensdom$ depends only on the ratio $\frac{b}{a}$, as rescaling
both $a$ and $b$ by the same factor amounts to an appropriate rescaling
of the Neumann domain together with the restriction of the eigenfunction
to it. The next theorem summarizes results on the spectral positions
of $\stardom$ and $\lensdom$ from \cite{BanEggTay_arXiv17} and
\cite{BanFaj_ahp16}.
\begin{thm}
\label{thm:Spectral-pos-star-and-lens}\cite{BanEggTay_arXiv17,BanFaj_ahp16}
\end{thm}
\begin{enumerate}
\item \label{enu:Lens-spectral-pos-unbounded}The set of spectral positions
of the \textbf{lens}-like domains $\left\{ N_{\lensdom}\left(\lambda_{a,b}\right)\right\} _{a,b}$
is unbounded. In particular, $N_{\lensdom}\left(\lambda_{a,b}\right)\rightarrow\infty$
for $\frac{a}{b}\rightarrow\infty$
\item \label{enu:star-spectral-pos-is-one}There exists $c>0$ such that
if $\frac{a}{b}>c$ then the spectral position of the \textbf{star}-like
domains is one, i.e., $N_{\stardom}\left(\lambda_{a,b}\right)=1$.
In addition, $\lambda_{a,b}$ is a simple eigenvalue of $\stardom$.
\end{enumerate}
\begin{rem*}
The condition $\frac{a}{b}>c$ in the second part of the theorem is
equivalent to the condition $\frac{m_{y}}{m_{x}}>c$ (see (\ref{eq:quantum-numbers-separable-efunc})).
As $m_{x},m_{y}\in\N$, this means that the claim in the second part
of the theorem is valid for a particular proportion of the separable
eigenfunctions on the torus. In particular, combining both parts of
the theorem, there is a range of values for $a,b$ for which $N_{\stardom}\left(\lambda_{a,b}\right)=1$,
but $N_{\lensdom}\left(\lambda_{a,b}\right)$ is as large as we wish.
\end{rem*}
The proofs of the two parts of this theorem are of different nature.
The proof of (\ref{enu:Lens-spectral-pos-unbounded}) appears in \cite{BanFaj_ahp16}.
It shows by means of contradiction that fixing the value of $a$ and
letting $b$ tend to zero the spectral positions $\{N_{\lensdom}\left(\lambda_{a,b}\right)\}_{a,b}$
cannot be bounded. This is done by proving that bounded spectral positions
would imply a too rapid growth of the number of Neumann domains. This
contradicts the actual growth of the number of Neumann domains, which
is explicitly known for those eigenfunctions.

The proof of (\ref{enu:star-spectral-pos-is-one}) appears in \cite{BanEggTay_arXiv17}.
It is based on three main ingredients. The first is the symmetry of
the domain $\stardom$ along a horizontal axis and a vertical axis.
The second is a non-standard rearrangement technique using a sector
as an intermediate domain \cite{LioPac_pams90,LioPacTri_iumj88} and
the third is the solution of a suitable geometric isoperimetric problem
with a constraint.

The motivation which stands behind Theorem \ref{thm:Spectral-pos-star-and-lens}
is the following. As already mentioned above, it was very natural
to believe that generically the spectral position equals one, just
as in the case of nodal domains. The first part of the theorem shows
that this belief is extremely violated in a particular case. The second
part somewhat revives this belief, by showing that this violation
which occurs for half of the Neumann domains is somewhat compensated
by the other half. We wonder whether this compensation holds for all
manifolds. For example, can it be that for any manifold, there exists
a constant $0<p\leq1$, such that each eigenfunction would have at
least a $p$ proportion of its Neumann domains with spectral position
equals to one? (see Lemma \ref{lem:Neumann-count-asymptotics}, where
a similar assumption is employed).

\section{Neumann domain count}

A wealth of results exists on the number of nodal domains. We start
this section by bounding the number of Neumann domains from below
in terms of the number of nodal domains. Denote the number of Neumann
domains of some eigenfunction $f$ by $\mu(f)$ and the number of
its nodal domains by $\nu(f)$. Observe that Theorem \ref{thm:topological-properties-manifolds},(\ref{enu:thm-topological-properties-manifolds-5})
implies that each Neumann domain intersects with exactly two nodal
domains (see discussion following Theorem \ref{thm:topological-properties-manifolds}).
This allows to conclude.
\begin{cor}
\label{cor:Neumann-count-bound-by-nodal-count}\cite{BanFaj_ahp16}
\begin{equation}
\mu(f)\geq\frac{1}{2}\nu(f).\label{eq:Neumann-bounded-by-nodal-manifold}
\end{equation}
\end{cor}
\noindent Next, we equip the Neumann lines with a graph structure
which we call the Neumann set graph. This allows to provide further
estimates on the number of the Neumann domains. Let $f$ be a Morse
function on a closed two-dimensional manifold and consider its Neumann
set graph obtained by taking the vertices ($V$) to be all critical
points, the edges ($E$) are the Neumann lines connecting critical
points and the faces ($F$) are the Neumann domains. Define the \emph{valency
of a critical point}, $\textrm{val}\left(\bs x\right)$, as the number
of Neumann lines which are connected to $x$.
\begin{prop}
\noindent \cite{BanFaj_ahp16} We have
\begin{equation}
\left|E\right|\leq4\left|\Sd\right|,\label{eq:upper_bound_on_edge_number}
\end{equation}

\begin{equation}
\mu(f)\leq2\left|\Sd\right|,\label{eq:upper_bound_for_Neumann_count}
\end{equation}
where $\Sd$ is the set of saddle points of $f$. If we further assume
a Morse-Smale function we get equalities in both (\ref{eq:upper_bound_on_edge_number})
and (\ref{eq:upper_bound_for_Neumann_count}). In addition we have
\begin{align}
\mu(f)=\frac{1}{2}\sum_{\bs x\in\Xt}\textrm{val}\left(\bs x\right)\geq\frac{1}{2}\left|\Xt\right|=\frac{1}{2}\left(\chi\left(M\right)+\left|\Sd\right|\right) & ,\label{eq:lower_bd_on_Neumann_count_in_terms_of_saddles}
\end{align}
where $\chi(M)$ is the Euler characteristic of the manifold.
\end{prop}
\noindent The proof of this proposition is done by combining Euler\textquoteright s
formula and Morse inequalities for the Neumann set graph.

\subsection*{The ratio $\frac{\mu_{n}}{n}$\label{subsec:mu_n-over-n}}

The most fundamental result for the nodal domain count is Courant's
bound $\frac{\nu_{n}}{n}\leq1$, where $\nu_{n}$ is the nodal count
of the $n^{\mathrm{th}}$ eigenfunction \cite{Cou_ngwgmp23}. Following
this, Pleijel had shown that $\limsup_{n\rightarrow\infty}\frac{\nu_{n}}{n}\leq\left(\frac{2}{j_{0,1}}\right)^{2}$,
where $j_{0,1}$ is the first zero of the $J_{0}$ Bessel function,
\cite{Pleijel_cpam56}. Many modern works concern the generalizations
or improvements of Pleijel's result, as well as the distribution of
the ratio $\frac{\nu_{n}}{n}$ \cite{BerHel_jst16,BluGnuSmi_prl02,Bourgain_imrn15,ChaHelHof_arXiv16,GnuLoi_jpa13,HelHof_spde15,Lena_arXiv16,Polterovich09,Steinerberger_ahp14}.
The study of the distribution of $\frac{\nu_{n}}{n}$ was initiated
in \cite{BluGnuSmi_prl02}. This distribution was presented there
for separable eigenfunctions on the rectangle and the disc. Later,
in \cite{GnuLoi_jpa13}, a more general calculation of the distribution
of $\frac{\nu_{n}}{n}$ was performed. It was done there for the Schr\"{o}dinger
operator on separable systems of any dimension.

In the following, we consider the analogous quantity, $\frac{\mu_{n}}{n}$,
the number of Neumann domains of the $n^{\mathrm{th}}$ eigenfunction
divided by $n$. We start by pointing out the connection between $\frac{\mu_{n}}{n},$
and the spectral position.
\begin{lem}
\label{lem:Neumann-count-asymptotics}Let $(M,g)$ be a two-dimensional,
connected, orientable and closed Riemannian manifold. Assume that
there exists $0<C\leq1$ such that
\begin{equation}
\sum_{\overset{\Omega\textrm{ s.t. }}{N_{\Omega}(\lambda_{n})=1}}\left|\Omega\right|>C\left|M\right|.\label{eq:Assumption-on-Spectral-position}
\end{equation}
 for all $\lambda_{n}$ in the spectrum of $M$, where the sum above
is over all Neumann domains (of an eigenfunction) of $\lambda_{n}$
whose spectral position equals one. Then

\begin{equation}
\liminf_{n\rightarrow\infty}\frac{\mu_{n}}{n}\geq C\left(\frac{2}{j}\right)^{2}.\label{eq:Neumann-count-ratio-lim-inf}
\end{equation}
\end{lem}
\begin{proof}
The Szeg\"o-Weinberger inequality \cite{Szego_jrma54,Weinberger_jrma56}
is $\lambda_{1}\left(\Omega\right)\left|\Omega\right|\leq\pi j^{2}$,
where $j$ is the first zero of the derivative of the $J_{1}$ Bessel
function. Consider an eigenfunction $f_{n}$ of $M$ corresponding
to an eigenvalue $\lambda_{n}$. For each Neumann domain $\Omega$
of $f_{n}$, for which $N_{\Omega}(\lambda_{n})=1$, we have $\lambda_{n}=\lambda_{1}(\Omega)$.
Combining the Szeg\"o-Weinberger inequality with the assumption in
the lemma gives
\[
\mu_{n}\pi j^{2}\geq\sum_{\overset{\Omega\textrm{ s.t. }}{N_{\Omega}(\lambda_{n})=1}}\pi j^{2}\geq\sum_{\overset{\Omega\textrm{ s.t. }}{N_{\Omega}(\lambda_{n})=1}}\lambda_{n}\left|\Omega\right|>C\lambda_{n}\left|M\right|.
\]
Applying Weyl asymptotics \cite{Weyl_1911} we get (\ref{eq:Neumann-count-ratio-lim-inf}).
\end{proof}
Such a result is interesting since it shows that the Neumann count
tends to infinity. Similar problems are investigated for the nodal
count. It was asked a few years ago by Hoffmann-Ostenhof whether $\limsup_{n\rightarrow\infty}\nu_{n}=\infty$
holds for any manifold \cite{Oberwolfach_report_12}. Following this,
Ghosh, Reznikov and Sarnak proved that the number of nodal domains
of Maass forms tends to infinity with the eigenvalue \cite{GhoRezSar_gafa13}.
Shortly afterwards, Jung and Zelditch have shown that for negatively
curved compact surfaces with some orientation-reversing isometric
involution, the number of nodal domains tends to infinity for a density
one sub-sequence of the eigenfunctions \cite{JunZel_jdg16}. Sequentially,
they improved upon this result by showing the same asymptotics for
non-positively surfaces without the need of an involution \cite{JunZel_mann16}.
The most recent result is by Zelditch who provided a logarithmic lower
bound for the nodal count of eigenfunctions on the first class of
manifolds mentioned above \cite{Zelditch_jst16}.

The validity of the inequality (\ref{eq:Neumann-count-ratio-lim-inf})
(and hence the validity of the assumption (\ref{eq:Assumption-on-Spectral-position}))
may be checked by investigating the distribution of $\frac{\mu_{n}}{n}$,
which is our next task.

We consider the separable eigenfunctions of the flat torus $\T$ with
fundamental domain $\R^{2}/\Z^{2}$. For those eigenfunctions we calculate
the limiting probability distribution of $\frac{\mu_{n}}{n}$.

Given a couple of natural numbers $m_{x},m_{y}\in\N$, we have that
\begin{equation}
f_{m_{x},m_{y}}(x,y)=\sin\left(2\pi m_{x}x\right)\cos\left(2\pi m_{y}y\right),\label{eq:Torus-eigenfunction-2}
\end{equation}
is a separable eigenfunction of the following eigenvalue
\begin{equation}
\lambda_{m_{x},m_{y}}:=4\pi^{2}\left(m_{x}^{2}+m_{y}^{2}\right),\label{eq:Torus-eigenvalue-2}
\end{equation}
(as in (\ref{eq:Torus-eigenvalue}),(\ref{eq:Torus-eigenfunction})).
Note that the functions $\cos\left(2\pi m_{x}x\right)\cos\left(2\pi m_{y}y\right)$,
$\cos\left(2\pi m_{x}x\right)\sin\left(2\pi m_{y}y\right)$, $\sin\left(2\pi m_{x}x\right)\sin\left(2\pi m_{y}y\right)$
together with (\ref{eq:Torus-eigenfunction-2}) are linearly independent
eigenfunctions which belong to the eigenvalue (\ref{eq:Torus-eigenvalue-2}).
The set of all those separable eigenfunctions for all possible values
of $m_{x},m_{y}\in\N$ form an orthogonal complete set of eigenfunctions
on $\mathbb{T}$.

We further note that the four eigenfunctions above which correspond
to a particular eigenvalue $\lambda_{m_{x},m_{y}}$ are equal on the
torus up to a translation. Hence, all four have the same number of
Neumann domains as $f_{m_{x},m_{y}}$ and we denote this number by
$\mu{}_{m_{x},m_{y}}$. With this we may define the following cumulative
distribution function
\begin{equation}
F_{\lambda}(c):=\frac{4}{N_{\T}(\lambda)}\left|\left\{ (m_{x},m_{y})\in\N^{2}~:~\lambda_{m_{x},m_{y}}<\lambda~~,~~\frac{\mu{}_{m_{x},m_{y}}}{N_{\T}(\lambda_{m_{x},m_{y}})}<c\right\} \right|,\label{eq:Cumulative-distribution-basic}
\end{equation}
where $N_{\T}(\lambda)$ is the spectral position of $\lambda$ in
the torus $\mathbb{T}$, as in (\ref{eq:Spectral-Position}), and
the factor $4$ stands for the four eigenfunctions which correspond
to $\lambda_{m_{x},m_{y}}$. In words, $F_{\lambda}(c)$ is the proportion
of the separable eigenfunctions with eigenvalue less than $\lambda$,
whose normalized Neumann count is smaller than $c$. Its limiting
distribution is given by the following.
\begin{prop}
\label{prop:Distribution-Neumann-Count-manifold}For $c<\frac{4}{\pi}$
\begin{equation}
\lim_{\lambda\to\infty}F_{\lambda}(c)=\frac{1}{2}\int_{0}^{c}\frac{1}{\sqrt{1-(\frac{\pi}{4}x)^{2}}}dx\label{eq:cumulative-dist-as-integral}
\end{equation}
and for $c\geq\frac{4}{\pi}$
\[
\lim_{\lambda\to\infty}F_{\lambda}(c)=1.
\]
\end{prop}
\begin{proof}
The proof consists of a reduction to a lattice counting problem, which
allows to derive the limiting distribution. First, observe that the
number of Neumann domains of $f_{m_{x},m_{y}}$ is $\mu{}_{m_{x},m_{y}}=8m_{x}m_{y}$.
This holds since $f_{m_{x},m_{y}}$ is Morse-Smale, so that there
is an equality in (\ref{eq:upper_bound_for_Neumann_count}), and the
number of saddle points of $f_{m_{x},m_{y}}$ is the number nodal
crossings which is easily shown to be $4m_{x}m_{y}$. The symmetry
between $m_{x}$ and $m_{y}$ in the expression for $\mu{}_{m_{x},m_{y}}$
motivate us to define the set
\[
W:=\left\{ (m_{x},m_{y})\in\N^{2}~:~m_{x}<m_{y}\right\} ,
\]
and observe
\begin{align}
\forall\lambda\quad\left|\left\{ (m_{x},m_{y})\in\N^{2}:~\lambda_{m_{x},m_{y}}<\lambda\right\} \right|= & 2\thinspace\left|\left\{ (m_{x},m_{y})\in W~:~\lambda_{m_{x},m_{y}}<\lambda\right\} \right|\nonumber \\
 & +\left|\left\{ (m_{x},m_{y})\in\N^{2}~:~m_{x}=m_{y}\textrm{ and }\lambda_{m_{x},m_{y}}<\lambda\right\} \right|\label{eq:Lattice-equivalence}
\end{align}
Plugging (\ref{eq:Lattice-equivalence}) in (\ref{eq:Cumulative-distribution-basic})
and taking the limit $\lambda\rightarrow\infty$ gives
\begin{equation}
\lim_{\lambda\to\infty}F_{\lambda}(c)=\lim_{\lambda\to\infty}\frac{8}{N_{\T}(\lambda)}\left|\left\{ (m_{x},m_{y})\in W~:~\lambda_{m_{x},m_{y}}<\lambda~~,~~\frac{\mu{}_{m_{x},m_{y}}}{N_{\T}(\lambda_{m_{x},m_{y}})}<c\right\} \right|,\label{eq:Cumulative-distribution-lattice}
\end{equation}
where we use the Weyl asymptotics, $\lim_{\lambda\rightarrow\infty}N_{\T}(\lambda)=\frac{\lambda}{4\pi}$
\cite{Weyl_1911} and that the second term in the right hand side
of (\ref{eq:Lattice-equivalence}) grows like $\sqrt{\lambda}$ and
hence drops when taking the limit.

We analyze (\ref{eq:Cumulative-distribution-lattice}) geometrically.
First, $N_{\T}(\lambda)$ counts the number of $\Z^{2}$ points with
non-zero coordinates, that lie inside a disc of radius $\sqrt{\lambda}$
around the origin. Hence, it may be written as
\begin{equation}
N_{\T}(\lambda_{m_{x},m_{y}})=\pi(m_{x}^{2}+m_{y}^{2})+Err(m_{x}^{2}+m_{y}^{2}),\label{eq:Counting-asymptotics}
\end{equation}
 where $Err(m_{x}^{2}+m_{y}^{2})=o(m_{x}^{2}+m_{y}^{2})$ \cite{Huxley_latticebook}.
In addition, the point $(m_{x},m_{y})\in W$ may be characterized
by the angle it makes with the $x$-axis, i.e., $\frac{m_{y}}{m_{x}}=\tan\theta_{m_{x,},m_{y}}$,
so that
\begin{equation}
\frac{2m_{x}m_{y}}{m_{x}^{2}+m_{y}^{2}}=2\cos\theta_{m_{x,},m_{y}}\cdot\sin\theta_{m_{x,},m_{y}}=\sin2\theta_{m_{x,},m_{y}}.\label{eq:theta_identity}
\end{equation}
With (\ref{eq:Counting-asymptotics}) and (\ref{eq:theta_identity})
we may write
\begin{align*}
\frac{\mu_{m_{x},m_{y}}}{N_{\T}(\lambda_{m_{x},m_{y}})} & =\frac{8m_{x}m_{y}}{\pi(m_{x}^{2}+m_{y}^{2})\left(1+Err(m_{x}^{2}+m_{y}^{2})/\pi(m_{x}^{2}+m_{y}^{2})\right)}\\
 & =\frac{1}{\left(1+Err(m_{x}^{2}+m_{y}^{2})/\pi(m_{x}^{2}+m_{y}^{2})\right)}\frac{4}{\pi}\cdot\sin2\theta_{m_{x,},m_{y}}.
\end{align*}

Let $\varepsilon>0$. Since $Err(m_{x}^{2}+m_{y}^{2})=o(m_{x}^{2}+m_{y}^{2})$,
there exists $\Lambda>0$ such that for all $(m_{x},m_{y})\in W$
satisfying $4\pi^{2}\left(m_{x}^{2}+m_{y}^{2}\right)>\Lambda$, the
following holds
\begin{equation}
\frac{1}{1+\varepsilon}\frac{4}{\pi}\sin2\theta_{m_{x,},m_{y}}<\frac{\mu_{m_{x},m_{y}}}{N_{\T}(\lambda_{m_{x},m_{y}})}<\frac{1}{1-\varepsilon}\frac{4}{\pi}\sin2\theta_{m_{x,},m_{y}}.\label{eq:normalized neumann count sandwich}
\end{equation}
The limiting cumulative distribution (\ref{eq:Cumulative-distribution-lattice})
may be slightly rewritten as
\begin{align*}
\lim_{\lambda\to\infty}F_{\lambda}(c) & =\lim_{\lambda\to\infty}\frac{8}{N_{\T}(\lambda)}\left|\left\{ (m_{x},m_{y})\in W~:~\Lambda<\lambda_{m_{x},m_{y}}<\lambda~~,~~\frac{\mu{}_{m_{x},m_{y}}}{N_{\T}(\lambda_{m_{x},m_{y}})}<c\right\} \right|,
\end{align*}
where the additional condition $\Lambda<\lambda_{m_{x},m_{y}}$ removes
only a finite number of points from the set and does not affect the
limit. We may now use (\ref{eq:normalized neumann count sandwich})
to get the following inequalities by set inclusion

\begin{align}
\lim_{\lambda\to\infty}F_{\lambda}(c) & \leq\label{eq:upper bound commulative}\\
\lim_{\lambda\to\infty}\frac{8}{N_{\T}(\lambda)} & \left|\left\{ (m_{x},m_{y})\in W\thinspace:\thinspace\Lambda<\lambda_{m_{x},m_{y}}<\lambda~\mathrm{and}~\theta_{m_{x,},m_{y}}<\frac{1}{2}\arcsin\left(\frac{\pi c(1+\varepsilon)}{4}\right)\right\} \right|,\nonumber
\end{align}
and

\begin{align}
\lim_{\lambda\to\infty}F_{\lambda}(c) & \geq\label{eq:lower bound commulative}\\
\lim_{\lambda\to\infty}\frac{8}{N_{\T}(\lambda)} & \left|\left\{ (m_{x},m_{y})\in W\thinspace:\thinspace\Lambda<\lambda_{m_{x},m_{y}}<\lambda~\mathrm{and}~\theta_{m_{x,},m_{y}}<\frac{1}{2}\arcsin\left(\frac{\pi c(1-\varepsilon)}{4}\right)\right\} \right|,\nonumber
\end{align}
where in the above we assume that $0\leq c<\frac{4}{\pi}$ and $\varepsilon$
is small enough so that $\pi c(1+\varepsilon)/4\leq1$, and in particular
$\arcsin(\pi c(1+\varepsilon)/4)$ is well defined.

We notice that the right hand sides of (\ref{eq:upper bound commulative})
and (\ref{eq:lower bound commulative}) correspond to counting integer
lattice points which are contained within a certain sector. This number
of points grows like the area of the corresponding sector \cite{Huxley_latticebook},
i.e.,
\begin{align}
\left|\left\{ (m_{x},m_{y})\in W~:~\Lambda<\lambda_{m_{x},m_{y}}<\lambda~~,~~\theta_{m_{x},n_{y}}<\frac{1}{2}\arcsin\left(\frac{\pi c(1\pm\varepsilon)}{4}\right)\right\} \right| & =\nonumber \\
=\underbrace{\frac{1}{4}\arcsin\left(\frac{\pi c(1\pm\varepsilon)}{4}\right)\frac{\lambda-\Lambda}{4\pi^{2}}}_{\text{area of a sector}}+o(\lambda).\label{eq:Sector-counting}
\end{align}

Plugging (\ref{eq:Sector-counting}) in the bounds (\ref{eq:upper bound commulative}),(\ref{eq:lower bound commulative})
and using (\ref{eq:Counting-asymptotics}) gives
\begin{align*}
\frac{2}{\pi}\arcsin\left(\frac{\pi c(1-\varepsilon)}{4}\right)\leq\lim_{\lambda\to\infty}F_{\lambda}(c)\leq\frac{2}{\pi}\arcsin\left(\frac{\pi c(1+\varepsilon)}{4}\right).
\end{align*}
As $\varepsilon>0$ is arbitrary we get
\begin{align*}
\forall c<\frac{4}{\pi}\quad\lim_{\lambda\to\infty}F_{\lambda}(c) & =\frac{2}{\pi}\arcsin\left(\frac{\pi c}{4}\right),\\
 & =\frac{1}{2}\int_{0}^{c}\frac{1}{\sqrt{1-(\frac{\pi}{4}x)^{2}}}dx,
\end{align*}
which proves (\ref{eq:cumulative-dist-as-integral}). Finally note
that we have $\lim_{c\rightarrow\frac{4}{\pi}}\lim_{\lambda\to\infty}F_{\lambda}(c)=1$,
and as $F_{\lambda}(c)$ is a cumulative distribution function we
get $\lim_{\lambda\to\infty}F_{\lambda}(c)=1$ for $c\geq\frac{4}{\pi}.$
\end{proof}
\begin{rem*}
The calculation in the proof above may be considered as a particular
case of those done in \cite{GnuLoi_jpa13}. The proof here is explicitly
tailored for the purpose of the current paper.
\end{rem*}
The next figure shows the probability distribution given in (\ref{eq:cumulative-dist-as-integral})
and compares it to a numerical examination of the probability distribution
of $\frac{\mu_{n}}{n}$ for the separable eigenfunctions on the torus.

\begin{figure}[h]
\centering{}\includegraphics[width=1\textwidth]{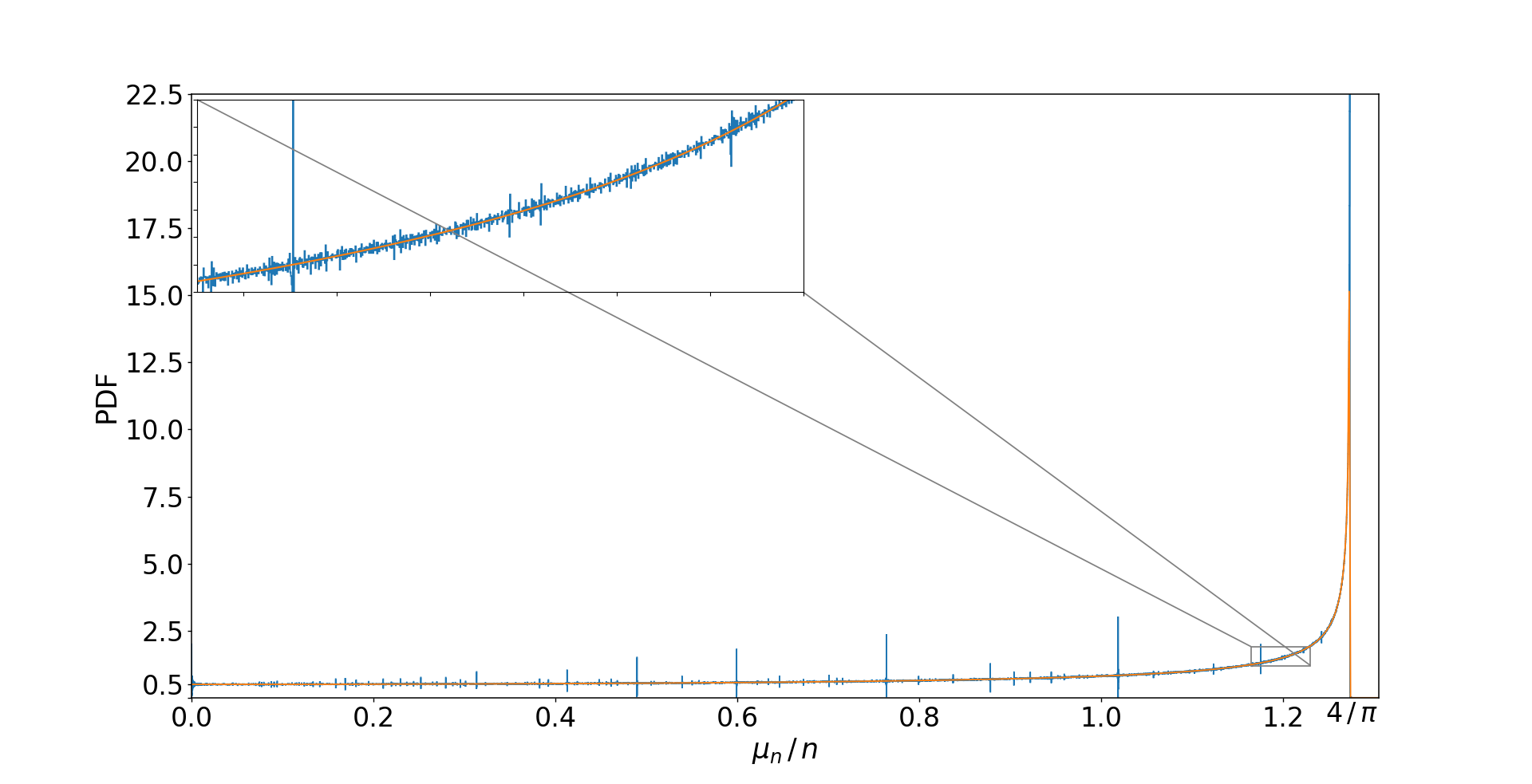} \caption{\emph{Orange curve}: the probability distribution of $\frac{\mu_{n}}{n}$
as given in (\ref{eq:cumulative-dist-as-integral}). \emph{Blue curve}:
a numerical calculation of this distribution as calculated for the
first $2\cdot10^{8}$ torus eigenfunctions.}
\label{fig:num5aa}
\end{figure}

Examining the $\frac{\mu_{n}}{n}$ distribution leads to the following.
First, we note that $\frac{\mu_{n}}{n}$ may get arbitrarily low values
for a positive proportion of the eigenfunctions. This is in contradiction
with (\ref{eq:Neumann-count-ratio-lim-inf}) and therefore we conclude
that the separable eigenfunctions on the flat torus do not satisfy
assumption (\ref{eq:Assumption-on-Spectral-position}) in Lemma \ref{lem:Neumann-count-asymptotics}.
Indeed, it can be checked directly in this case that the total area
of all star domains goes to zero as the eigenvalue $\lambda_{a,b}$
tends to infinity. Therefore their total area does not satisfy (\ref{eq:Assumption-on-Spectral-position}).

Turning our attention to the higher values of $\frac{\mu_{n}}{n}$
we notice that $\frac{\mu_{n}}{n}>1$ for a positive proportion of
the eigenfunction. This means that the strict Courant bound does not
apply to the Neumann domain count. Yet, we ask whether an upper bound
of the form $\mu_{n}\leq h(n)$ exists, with $h$ being possibly a
linear function. If so, a possible upper bound might be suggested
by the separable eigenfunctions, for which $h(n)=\frac{4}{\pi}n$,
serves as a bound. Furthermore, the Courant-like bound on the Neumann
count for metric graphs (\ref{eq:Neumann-Surplus-bounds}), to be
discussed in the next part, is a suggestive evidence for the existence
of a Courant-like bound in the case of manifolds\footnote{While writing this manuscript we became aware of a work in progress
by Buhovski, Logunov, Nazarov and M. Sodin, which might disprove the
existence of such a bound.}.

\vspace{24mm}

\part{Neumann domains on metric graphs}

\vspace{8mm}

\section{Definitions\label{sec:Definitions-graphs}}

\subsection{Discrete graphs and graph topologies\label{subsec:discrete_graphs_introduction}}

We denote by $\Gamma=\left(\V,\E\right)$ a connected graph with finite
sets of vertices $\V$ and edges $\E$. We allow the graph edges to
connect either two distinct vertices or a vertex to itself. In the
latter case, such an edge is called a loop.

For a vertex $v\in\V$, its \emph{degree, $d_{v}$,} equals the number
of edges connected to it. The set of graph vertices of degree one
turns out to be useful and we denote it by
\[
\partial\Gamma:=\left\{ v\in\V\thinspace:\,d_{v}=1\right\} .
\]
We call the vertices in $\partial\Gamma$, \emph{boundary vertices}
and the rest of the vertices, $\V\backslash\partial\Gamma$, are called
\emph{interior vertices}.

An important topological quantity of graphs is the first Betti number
(dimension of the first homology group) given, for a connected graph,
by
\begin{equation}
\beta:=\left|\E\right|-\left|\V\right|+1.\label{eq:Betti_number}
\end{equation}
The value of $\beta$ is the number cycles needed to span the space
of cycles on the graph. By definition a graph is simply connected
when $\beta=0$, and such a graph is called a tree graph. Two particular
examples of trees are star graphs and path graphs. A star graph is
a graph with one interior vertex which is connected by edges to the
other $\left|\V\right|-1$ boundary vertices. A path graph is a connected
graph with two boundary vertices and $\left|\V\right|-2$ interior
vertices which are all of degree two. The path graph which shows up
later in this paper is the simplest graph of only two vertices connected
by a single edge.

\subsection{Spectral theory of metric graphs\label{subsec:Spectral-theory-of-Quantum-Graphs}}

A \emph{metric graph} is a discrete graph for which each edge, $e\in\E$,
is identified with a one-dimensional interval $[0,L_{e}]$ of a positive
finite length $L_{e}$. We assign to each edge $e\in\E$ a coordinate,
$x_{e}$, which measures the distance along the edge from one of the
two boundary vertices of $e$.

A function on the graph is described by its restrictions to the edges,
$\left\{ \left.f\right|_{e}\right\} _{e\in\E}$, where $\left.f\right|_{e}:\left[0,L_{e}\right]\rightarrow\C$.
We equip the metric graphs with a self-adjoint differential operator,
\begin{equation}
-\Delta\ :\ \left.f\right|_{e}(x_{e})\mapsto-\frac{\ud^{2}}{\ud x_{e}^{2}}\left.f\right|_{e}\left(x_{e}\right),\label{eq:metric_Schroedinger}
\end{equation}
which is the Laplacian\footnote{More general operators appear in the literature. See for example \cite{BerKuc_graphs,GnuSmi_ap06}.}.
It is most common to call this setting of a metric graph and an operator
by the name quantum graph.

To complete the definition of the operator we need to specify its
domain. We consider functions which belong to the following direct
sum of Sobolev spaces
\begin{equation}
H^{2}(\metgraph):=\bigoplus_{e\in\E}H^{2}([0,L_{e}])\ .\label{eq:Sobolev_on_graph}
\end{equation}
In addition we require some matching conditions on the graph vertices.
A function $f\in H^{2}(\metgraph)$ is said to satisfy the Neumann
vertex conditions at a vertex $v$ if
\begin{enumerate}
\item $f$ is continuous at $v\in\V$, i.e.,
\begin{equation}
\forall e_{1},e_{2}\in\E_{v\,\,\,\,\,}\left.f\right|_{e_{1}}(0)=\left.f\right|_{e_{2}}(0),\label{eq:Neumann_continuity}
\end{equation}
where $\E_{v}$ is the set of edges connected to $v$, and for all
$e\in\E_{v}$, $x_{e}=0$ at $v$.
\item The outgoing derivatives of $f$ at $v$ satisfy
\begin{equation}
\sum_{e\in\E_{v}}\left.\frac{\ud f}{\ud x_{e}}\right|_{e}\left(0\right)=0.\label{eq:Neumann_deriv_conditions}
\end{equation}
\end{enumerate}
Requiring these conditions at each vertex leads to the operator \eqref{eq:metric_Schroedinger}
being self-adjoint and its spectrum being real and bounded from below
\cite{BerKuc_graphs}. In addition, since we only consider compact
graphs, the spectrum is discrete. We number the eigenvalues in the
ascending order and denote them by $\left\{ \lambda_{n}\right\} _{n=0}^{\infty}$
and their corresponding eigenfunctions by $\left\{ f_{n}\right\} _{n=0}^{\infty}$.
As the operator is both real and self-adjoint, we may choose the eigenfunctions
to be real, which we will always do.

In this paper, we only consider graphs whose vertex conditions are
Neumann at all vertices, and call those \emph{standard graphs.} A
special attention should be given to vertices of degree two. Introducing
such a vertex at the interior of an existing edge (thus splitting
this edge into two) and requiring Neumann conditions at this vertex
does not change the eigenvalues and eigenfunctions of the graph. The
same holds when removing a degree two vertex and uniting two existing
edges into one. This spectral invariance allows us to assume in the
following that standard graphs do not have any vertices of degree
two. Furthermore, the only graph, all of whose vertices are of degree
two (or equivalently has no vertices at all) is the single loop graph.
We assume throughout the paper that our graphs are different than
the single loop graph and call those \emph{nontrivial graphs}.

The spectrum of a standard graph is non-negative, which means that
we may represent the spectrum by the non-negative square roots of
the eigenvalues, $k_{n}=\sqrt{\lambda_{n}}$. For convenience we abuse
terminology and call also $\left\{ k_{n}\right\} _{n=0}^{\infty}$
the eigenvalues of the graph. Most of the results and proofs in this
part are expressed in terms of those eigenvalues. A Neumann graph
has $k_{0}=0$ with multiplicity which equals the number of graph
components. The common convention is that if an eigenvalue is degenerate
(i.e. non simple) it appears more than once in the sequence $\left\{ k_{n}\right\} _{n=0}^{\infty}$.
For any such degenerate eigenvalue, we pick a basis for its eigenspace
and all members of this basis appear in the sequence $\left\{ f_{n}\right\} _{n=0}^{\infty}$.
Obviously, this makes the choice of the sequence $\left\{ f_{n}\right\} _{n=0}^{\infty}$
non unique. It is important to note that all the statements to follow
hold for any choice of $\left\{ f_{n}\right\} _{n=0}^{\infty}$.

\subsection{Neumann points and Neumann domains}

For metric graphs, the nodal point set of a function is the set of
points at which the function vanishes. Removing the nodal point set
from the graph, splits it into connected components and those are
called nodal domains. The Neumann set and Neumann domains are similarly
defined, but before doing so we need to restrict to particular classes
of functions.
\begin{defn}
\label{def:Morse-and-Generic}Let $\Gamma$ be a nontrivial standard
graph and $f$ be an eigenfunction of $\Gamma$.
\end{defn}
\begin{enumerate}
\item We call $f$ a \emph{Morse eigenfunction} if for each edge $e$, $\left.f\right|_{e}$
is a Morse function. Namely, at no point in the interior of $e$ both
the first and the second derivatives of $f$ vanish.
\item We call an eigenfunction $f$ \emph{generic} if it is a Morse eigenfunction
and in addition satisfies all of the following:
\begin{enumerate}
\item $f$ corresponds to a simple eigenvalue.
\item $f$ does not vanish at any vertex.
\item $f$ has no extremal points at interior vertices.
\end{enumerate}
\end{enumerate}
An equivalent characterization of a Morse eigenfunction is
\begin{lem}
\label{lem:Morse-equiv-non-zero}Let $f$ be a non-constant eigenfunction.
$f$ is Morse if and only if there exists no edge $e$ such that $\left.f\right|_{e}\equiv0$.
\end{lem}
\begin{proof}
First, observe that a non-constant eigenfunction of the Laplacian
vanishes at an interior point of an edge if and only if the second
derivative vanishes at that point. Therefore, if $f$ is a Morse eigenfunction
then there is no interior point at which both the function and its
derivative vanish. This means that a Morse eigenfunction cannot vanish
entirely at a graph edge. As for the converse, if $f$ is a non-Morse
eigenfunction then there exists $\bs x$, an interior point of an
edge $e$, such that $f|_{e}'(\bs x)=f|_{e}''(\bs x)=0$. By the same
argument as above, this means that either $f|_{e}(\bs x)=0$ or $f|_{e}$
is the constant eigenfunction. The vanishing of $f|_{e}$ and its
first derivative at the same point, together with $f|_{e}$ being
a solution of an ordinary differential equation of second order implies
$\left.f\right|_{e}\equiv0$.
\end{proof}
We complement this lemma and note that the constant eigenfunction,
corresponding to $k_{0}=0$ is not a Morse function. This, together
with the lemma, implies that a Morse eigenfunction may vanish only
at isolated points of the graph; the same holds for its derivative.
This quality allows the following.
\begin{defn}
\label{def:Neumann-point-Neumann-Domain}Let $f$ be a Morse eigenfunction.
\end{defn}
\begin{enumerate}
\item A Neumann point of $f$ is an extremal point (maximum or minimum)
not located at a boundary vertex. Namely, the set of Neumann points
is
\begin{equation}
\Nd:=\left\{ \bs x\in\Gamma\backslash\partial\Gamma\thinspace:\thinspace\bs x\textrm{ is an extremal point of }f\right\} .\label{eq:Neumann-point-set}
\end{equation}
Note that we reuse here the notation for the Neumann lines in the
manifold case, (\ref{eq:Neumann-line-set}).
\item A Neumann domain of $f$ is a closure of a connected component of
$\Gamma\backslash\Nd$. The closure is done by adding vertices of
degree one at the open endpoints of the connected component.
\end{enumerate}
Figure \ref{fig:Neumann domain example} shows the Neumann point and
Neumann domains of a particular eigenfunction.
\begin{figure}[h]
\centering{}(i)$\quad$\includegraphics[width=0.2\textwidth]{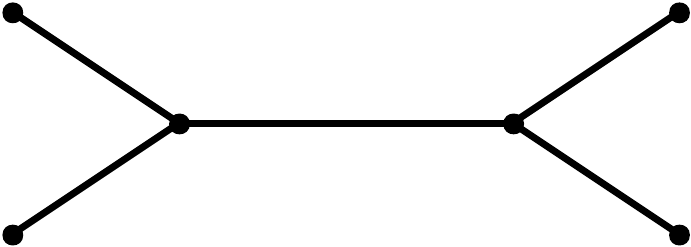}
\hfill{}(ii)\includegraphics[width=0.3\textwidth]{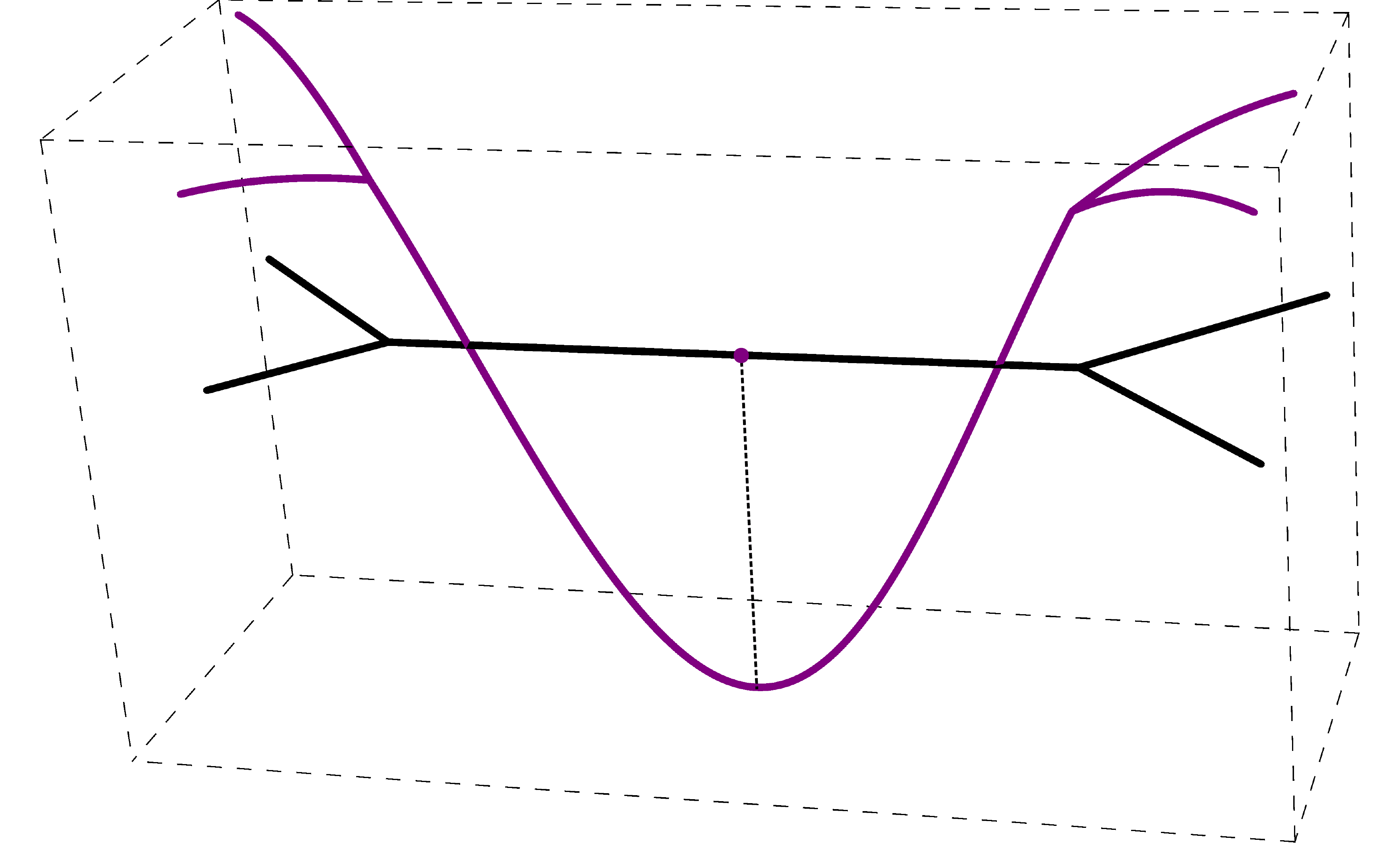} \hfill{}(iii)\includegraphics[width=0.3\textwidth]{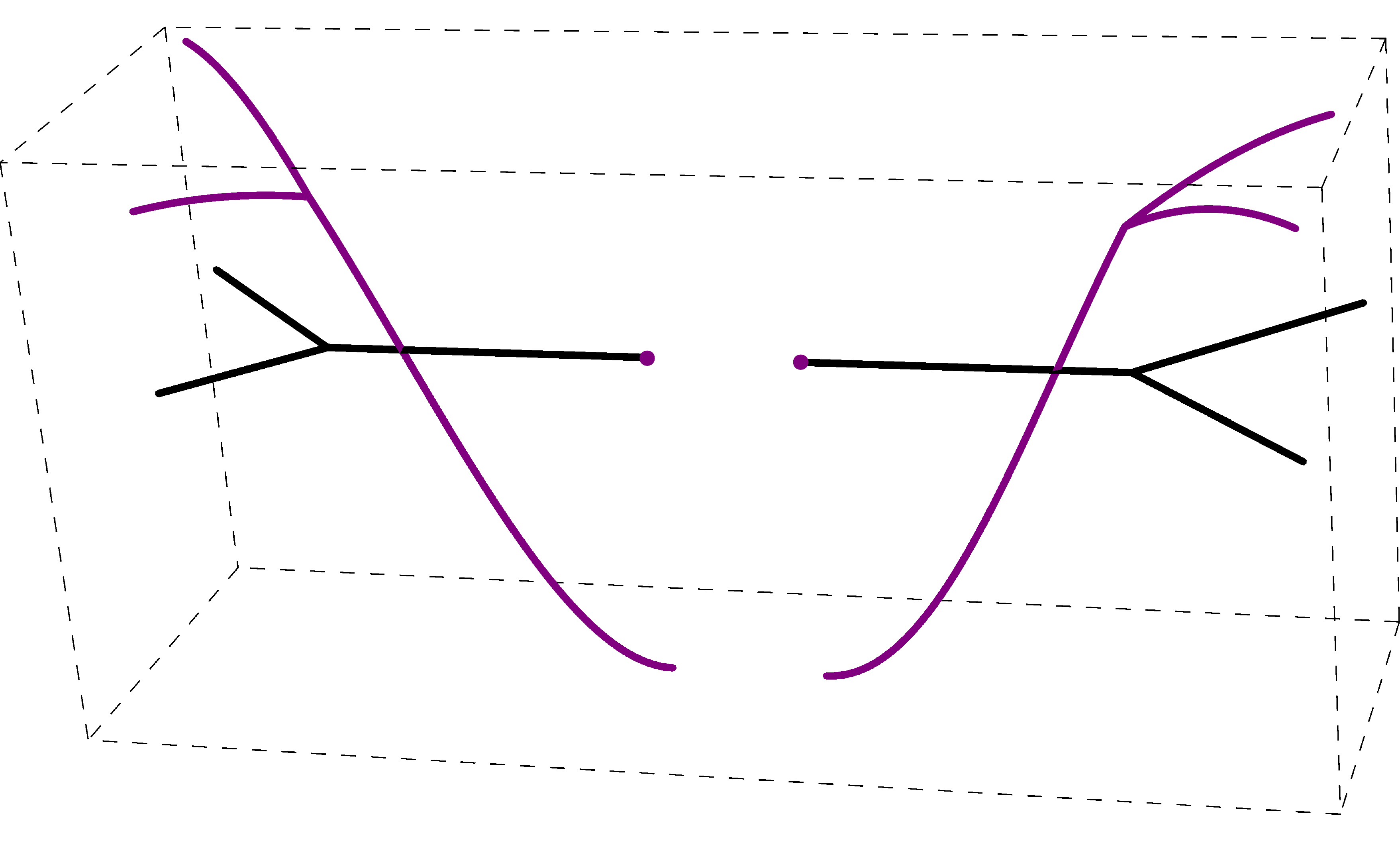}
\caption{(i) A graph $\Gamma$ (ii) An eigenfunction $f$ of $\Gamma$, with
its single Neumann point marked (iii) A decomposition of $\Gamma$
into the Neumann domains of $f$.}
\label{fig:Neumann domain example}
\end{figure}

\begin{rem*}
~
\begin{enumerate}
\item The definition implies that a Neumann point is either a point $\bs x\in\Gamma\backslash\V$
at some interior of an edge such that $f'(\bs x)=0$, or it is a vertex
$v\in\V$ such that all outgoing derivatives of $f$ at that vertex
vanish. The latter possibility does not occur if $f$ is generic.
\item From the proof of Lemma \ref{lem:Morse-equiv-non-zero} we learn that
no point can be both a nodal point and a Neumann point.
\end{enumerate}
\end{rem*}
All the results to follow concerning Neumann points and Neumann domains
are stated for either Morse or generic eigenfunctions. We start by
stating what proportion of the eigenfunctions are Morse and which
proportion of the Morse ones are generic. In order to do so, we need
to assume that the set of edge lengths is linearly independent over
the field $\Q$. We call such lengths \emph{rationally independent}
and we will assume this for the graph edge lengths in some of the
propositions to follow.
\begin{prop}
\cite{AloBan_Neumann,AloBanBer_cmp18}\label{prop:Proportion-Morse-generic}
Let $\Gamma$ be a nontrivial standard graph, with rationally independent
edge lengths $\left\{ L_{e}\right\} _{e\in\E}$. Let $\left\{ f_{n}\right\} _{n=0}^{\infty}$
be a complete set of eigenfunctions of $\Gamma$.
\begin{enumerate}
\item \label{enu:prop-Proportion-Morse-generic-1}The proportion of Morse
eigenfunctions is given by
\begin{equation}
\dens:=\lim_{N\rightarrow\infty}\frac{\left|\left\{ n\le N\,:\,f_{n}\,\textrm{is Morse}\right\} \right|}{N}=1-\frac{1}{2}\frac{\sum_{e\in\E_{L}}L_{e}}{\sum_{e\in\E}L_{e}},\label{eq:proportion-of-Morse}
\end{equation}
where $\E$ is the set of graph edges and $\E_{L}$ is a subset of
$\E$ consisting of all edges which form loops (edges which connect
a vertex to itself).
\item \label{enu:prop-Proportion-Morse-generic-2}The proportion of generic
eigenfunctions out of the Morse ones is
\begin{equation}
\lim_{N\rightarrow\infty}\frac{\left|\left\{ n\le N\,:\,f_{n}\,\textrm{is generic}\right\} \right|}{\left|\left\{ n\le N\,:\,f_{n}\,\textrm{is Morse}\right\} \right|}=\frac{1}{\dens}\lim_{N\rightarrow\infty}\frac{\left|\left\{ n\le N\,:\,f_{n}\,\textrm{is generic}\right\} \right|}{N}=1.\label{eq:proportion-of-generic}
\end{equation}
Namely, almost all Morse eigenfunctions are generic.
\end{enumerate}
\end{prop}
\begin{rem*}
~
\begin{enumerate}
\item The limits in (\ref{eq:proportion-of-Morse}) and (\ref{eq:proportion-of-generic})
exist even without assuming that the edge lengths are rationally independent.
This assumption is needed to obtain the exact values of those limits.
\item From the proposition we get that at least half of the eigenfunctions
are Morse, and if a graph has no loops, almost all eigenfunctions
are Morse and generic.
\item The proof of (\ref{eq:proportion-of-Morse}) is similar to the proof
of proposition A.1 in \cite{AloBanBer_cmp18}. The proof of (\ref{eq:proportion-of-generic})
appears in \cite{AloBan_Neumann}.
\end{enumerate}
\end{rem*}

\section{Topology of $\Omega$ and topography of $\left.f\right|_{\Omega}$\label{sec:Topology-and-Topography-Graphs}}

Let $\Gamma$ be a nontrivial standard graph and $f$ an eigenfunction
of $\Gamma$ corresponding to the eigenvalue $k$. Formally, every
Neumann domain $\Omega$ of $f$ may be considered as a subgraph of
$\Gamma$, if we add degree two vertices to $\Gamma$ at all the Neumann
points of $f$ (see discussion on those vertices in Section \ref{subsec:Spectral-theory-of-Quantum-Graphs}).
In particular, a Neumann domain is a closed set (by Definition \ref{def:Neumann-point-Neumann-Domain}).
This difference from the manifold case (where Neumann domains are
open sets) is technical and serves our need to consider $\Omega$
as a metric graph on its own. Being a metric graph, we take the usual
Laplacian on $\Omega$ and impose Neumann vertex conditions at all
of its vertices, so that $\Omega$ is considered as a standard graph.
Note that the restriction of $f|_{\Omega}$ to the edges of $\Omega$
trivially satisfies $f''=-k^{2}f$. It also obeys Neumann vertex conditions
at all vertices of $\Omega$, as each vertex is either a vertex of
$\Gamma$ or a point $\bs x\in\Gamma$ in an interior of an edge for
which $f'(\bs x)=0$. This gives the following, which is analogous
to Lemma \ref{lem:Restriction-is-Neumann-eigenfunction-manifold}.
\begin{lem}
\label{lem:Restriction-is-Neumann-eigenfunction-graph}$\left.f\right|_{\Omega}$
is an eigenfunction of the standard graph $\Omega$ and corresponds
to the eigenvalue $k$.
\end{lem}
\begin{rem*}
Furthermore, it can be proved that if $f$ is a generic eigenfunction
and $\Omega$ is a tree graph then $f|_{\Omega}$ is also generic
\cite{AloBan_Neumann}.
\end{rem*}

\subsection{Possible topologies for Neumann domains}

In this subsection we discuss which graphs may be obtained as a Neumann
domain. The next lemma shows that if we consider an eigenfunction,
$f$, whose eigenvalue is high enough, each of its Neumann domains
is either a path graph or a star graph. A star Neumann domain contains
an interior vertex of the graph, and a path Neumann domain is contained
in a single edge of the graph (see Figure \ref{fig:Neumann domain example-1}).
\begin{figure}[h]
\centering{}(i)$\quad$\includegraphics[width=0.3\textwidth]{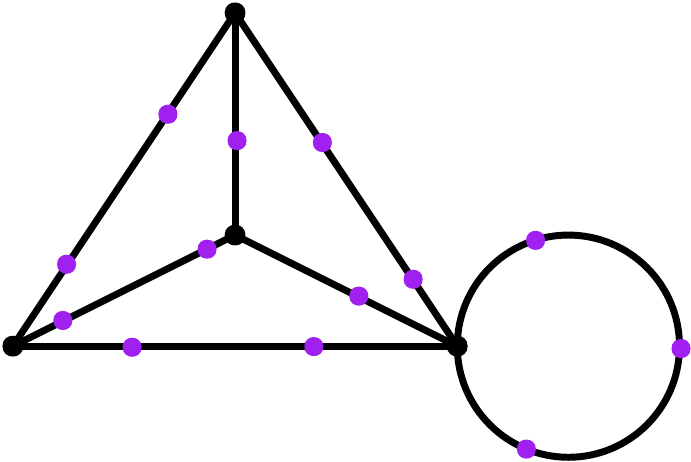}$\quad$(ii)$\quad$\includegraphics[width=0.3\textwidth]{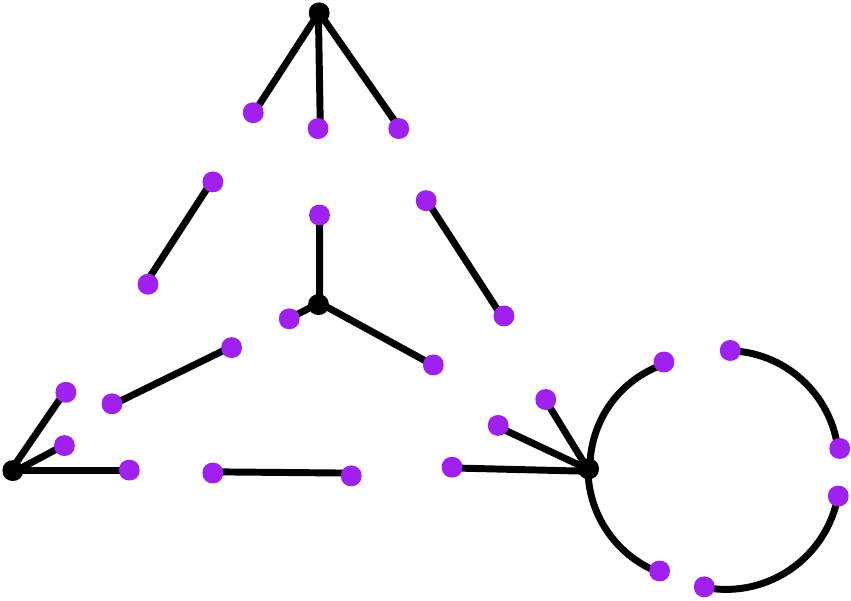}
\caption{(i) A graph $\Gamma$ with Neumann points (in purple) of a given eigenfunction
(ii) The decomposition of the graph to the corresponding Neumann domains.}
\label{fig:Neumann domain example-1}
\end{figure}

\begin{lem}
\label{lem:High-eigenvalue}Let $\Gamma$ be a nontrivial standard
graph. Let $f$ be an eigenfunction corresponding to an eigenvalue
$k>\frac{\pi}{L_{min}}$, where $L_{min}$ is the minimal edge length
of $\Gamma$. Let $\Omega$ be a Neumann domain of $f$.
\end{lem}
\begin{enumerate}
\item If $\Omega$ contains a vertex $v\in\V$ of degree $d_{v}>2$ then
$\Omega$ is a \emph{star graph}\textbf{ }with $deg\left(v\right)$
edges.
\item If $\Omega$ does not contain a vertex $v\in\V$ of degree $d_{v}>2$
then\textbf{ }$\Omega$ is a \emph{path graph, of length $\frac{\pi}{k}$.}
\end{enumerate}
\begin{proof}
For any edge $e\in\E$ we have that $f|_{e}\left(x\right)=B_{e}\cos\left(kx+\varphi_{e}\right)$,
where $B_{e},\varphi_{e}$ are some edge dependent real parameters.
This together with $k>\frac{\pi}{L_{min}}$ implies that the derivative
of $f$ vanishes at least once at the interior of each edge. Hence,
the set of Neumann points, $\Nd$ contains at least one point on each
edge. It follows that each Neumann domain contains at most one vertex
of $\Gamma$. Thus, there are two types of Neumann domains: if a Neumann
domain, $\Omega$, contains a vertex with $\deg v>2$ then $\Omega$
is a \emph{star graph}, whose number of edges is $d_{v}$; otherwise
$\Omega$ is a \emph{path graph}. A Neumann domain which is a path
graph can be parameterized as $\Omega=[0,l]$. Since $f'(0)=0$ we
get that $f|_{\Omega}\left(x\right)=\cos\left(kx\right)$ up to a
multiplicative constant . Using $f'(l)=0$ and that $f'$ does not
vanish in the interior of $\Omega$ we conclude $l=\frac{\pi}{k}$.
\end{proof}
\begin{rem*}
Only finitely many eigenvalues do not satisfy the condition $k>\frac{\pi}{L_{min}}$
in the lemma. The number of those eigenvalues is bounded by
\[
\left|\left\{ n\in\N\,:\,0\leq k_{n}\le\frac{\pi}{L_{min}}\right\} \right|~\le~2\frac{\left|\Gamma\right|}{L_{min}},
\]
where $\left|\Gamma\right|=\sum_{e\in\E}L_{e}$ is the total sum of
all edge lengths of $\Gamma$. This can be shown using
\begin{equation}
\forall n\in\N,\quad k_{n}\ge\frac{\pi}{2\left|\Gamma\right|}\left(n+1\right),\label{eq:Fried-bound}
\end{equation}
which is the statement of Theorem 1 in \cite{Fri_aif05}.
\end{rem*}
To complement the lemma above, we note that there are also Neumann
domains which are not simply connected. Indeed, consider the graph
$\Gamma$ depicted in Figure \ref{fig:Non-simply-connected-ND}(i).
It has an eigenfunction with no Neumann points, so that the eigenfunction
has a single Neumann domain which is the whole of $\Gamma$ and in
particular, it is not simply connected (Figure \ref{fig:Non-simply-connected-ND}(ii)).

\begin{figure}[h]
\centering{}(i) $\quad$\includegraphics[width=0.3\textwidth]{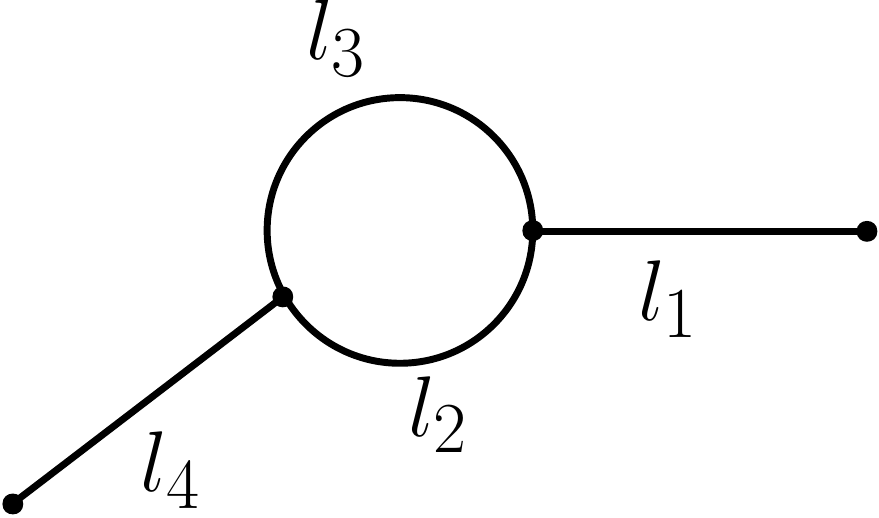}
\hfill{}(ii)$\quad$\includegraphics[width=0.4\textwidth]{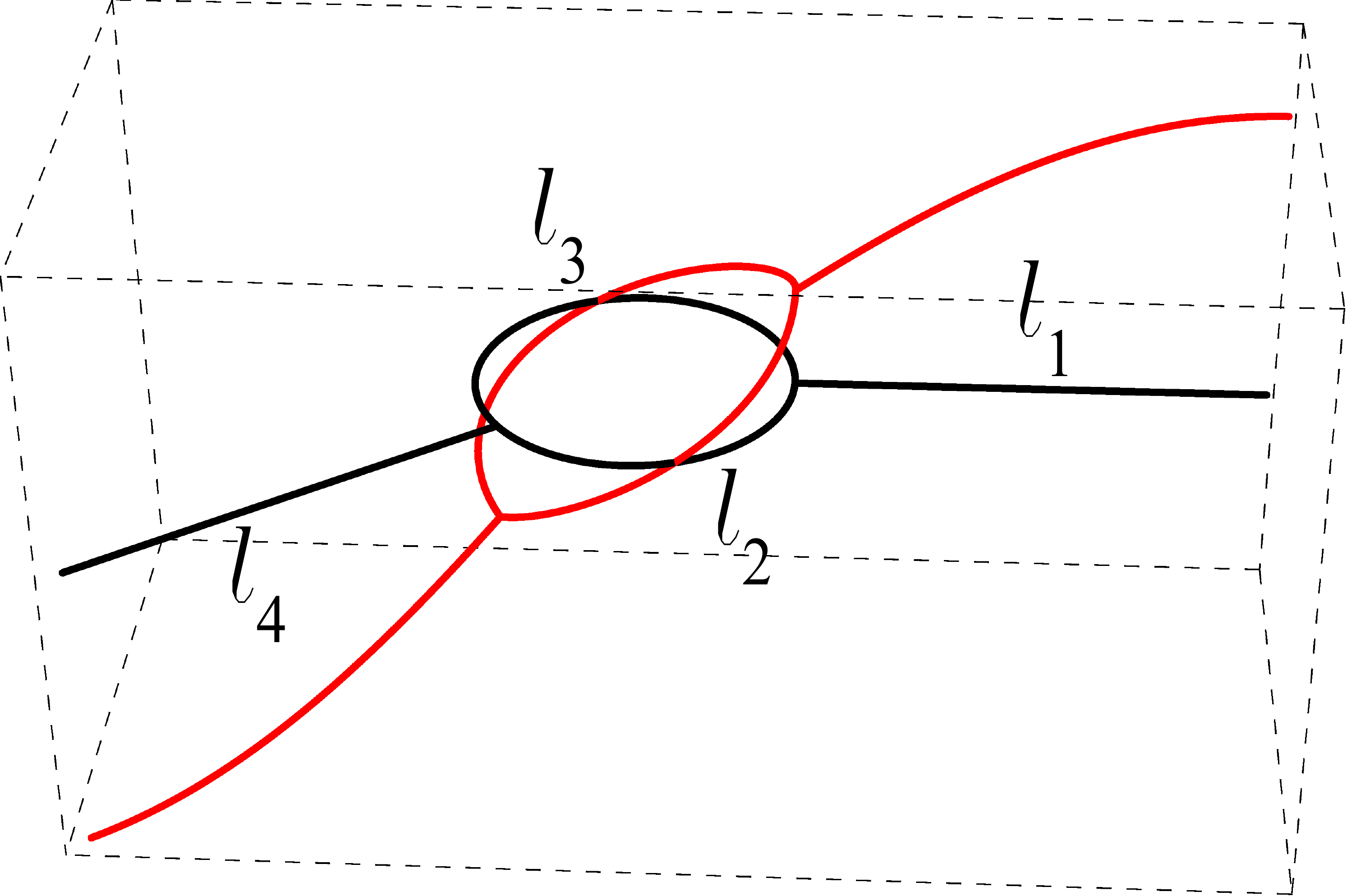}
\caption{(i) A graph $\Gamma$ with (ii) An eigenfunction whose single Neumann
domain is not simply connected.}
\label{fig:Non-simply-connected-ND}
\end{figure}

\subsection{Critical points and nodal points - number and position}

In the following we consider the critical points and nodal points
of $f|_{\Omega}$. Note that, by definition, a Morse function on a
one dimensional interval cannot have a saddle point. Hence, all critical
points of a Morse eigenfunction of a graph are extremal points. We
reuse the notations from the manifold part: $\Xt$ for extremal points
of $f$ and $\Max$ ($\Min$) for maxima (minima). Denote by $\phi(f|_{\Omega})$
the number of nodal points of $f|_{\Omega}$, by $E_{\Omega}$ the
number of edges of $\Omega$, by $V_{\Omega}$ the number of its vertices,
and by $\partial\Omega$ the vertices of $\Omega$ which are of degree
one.
\begin{prop}
\cite{AloBan_Neumann}\label{prop:Critical-and-nodal-points-graph}
Let $f$ be a generic eigenfunction and $\Omega$ a Neumann domain
of $f$. Then
\end{prop}
\begin{enumerate}
\item \label{enu:Critical-and-nodal-points-graph-1} The extremal points
of $f|_{\Omega}$ , which are located on $\Omega$ are exactly the
boundary of $\Omega$, i.e., $\Xt\cap\Omega=\partial\Omega$
\item \label{enu:Critical-and-nodal-points-graph-2}$1\leq\left|\Max\cap\partial\Omega\right|\le\left|\partial\Omega\right|-1$
(and the same bounds for $\left|\Min\cap\partial\Omega\right|$).
\item \label{enu:Critical-and-nodal-points-graph-3}$1\le\phi(f|_{\Omega})\le E_{\Omega}-V_{\Omega}+\left|\partial\Omega\right|$.
\end{enumerate}
\begin{rem*}
Note that when $\Omega$ is a path graph the proposition implies that
it has exactly one maximum, one minimum and one nodal point. Also,
when $\Omega$ is a tree graph, the last part of the proposition gives
$1\le\phi(f|_{\Omega})\le\left|\partial\Omega\right|-1$.
\end{rem*}

\section{Geometry of $\Omega$}

Similarly to the manifold case we use the normalized area to perimeter
ratio to quantify the geometry of a Neumann domain. The following
is to be compared with Definition \ref{def:Area-to-Perimeter-Manifolds}.
\begin{defn}
\label{def:Area-to-perimeter-graphs} Let $f$ be a Morse eigenfunction
corresponding to the eigenvalue $k$. Let $\Omega$ be a Neumann domain
of $f$, whose edge lengths are $\{l_{j}\}_{j=1}^{E_{\Omega}}$. We
define the normalized area to perimeter ratio of $\Omega$ to be
\[
\rho\left(\Omega\right):=\frac{\left|\Omega\right|}{\left|\d\Omega\right|}k,
\]
where $\left|\Omega\right|=\sum_{j=1}^{E_{\Omega}}l_{j}$ and $\left|\partial\Omega\right|$
is the number of boundary vertices of $\Omega$.
\end{defn}
For graphs we are able to obtain global bounds on $\rho(\Omega)$.
\begin{prop}
\cite{AloBan_Neumann}\label{prop:rho-Bounds-graph} Let $\Omega$
be a Neumann domain. We have
\begin{equation}
\frac{1}{\left|\partial\Omega\right|}\le\frac{\rho\left(\Omega\right)}{\pi}\le\frac{E_{\Omega}}{\left|\partial\Omega\right|}.\label{eq:rho-Bounds-graph}
\end{equation}
If $\Omega$ is a star graph then we have a better upper bound $\frac{\rho\left(\Omega\right)}{\pi}\le1-\frac{1}{\left|\partial\Omega\right|}$.\\
If $\Omega$ is a path graph then $\rho\left(\Omega\right)=\frac{\pi}{2}$.
\end{prop}
Next, we study the probability distribution of $\rho$. We find that
for this purpose, it is useful to consider separately only the Neumann
domains containing a particular vertex. Let $\Gamma$ be a nontrivial
standard graph and let $f_{n}$ be its $n^{th}$ eigenfunction. Assume
that $f_{n}$ is generic. Then, for any vertex $v\in\V$ there is
a unique Neumann domain of $f_{n}$ which contains $v$ and we denote
it by $\Omega_{n}^{\left(v\right)}$.
\begin{prop}
\cite{AloBan_Neumann}\label{prop:rho-distribution} Let $v\in\V$
of degree $d_{v}>2$. The value of $\frac{1}{\pi}\rho$ on $\{\Omega_{n}^{\left(v\right)}\}_{n=1}^{\infty}$
is distributed according to
\begin{equation}
\lim_{N\rightarrow\infty}\frac{\left|\left\{ n\le N\,:\,f_{n}\,\textrm{is generic and }\frac{1}{\pi}\rho\left(\Omega_{n}^{\left(v\right)}\right)\in\left(a,b\right)\right\} \right|}{\left|\left\{ n\le N\,:\,f_{n}\,\textrm{is generic}\right\} \right|}=\int_{a}^{b}\zeta^{(v)}(x)\thinspace\mathrm{d}x,\label{eq:rho-distribution}
\end{equation}
where $\zeta^{(v)}$ is a probability distribution supported on $[\frac{1}{d_{v}},1-\frac{1}{d_{v}}]$.

Furthermore, it is symmetric around $\frac{1}{2}$, i.e. $\zeta^{(v)}(x)=\zeta^{(v)}(1-x)$.
\end{prop}
\begin{rem*}
If $d_{v}=1$ then $\Omega_{n}^{\left(v\right)}$ is a path graph
for all $n$, so that by Proposition \ref{prop:rho-Bounds-graph}
we get that $\zeta^{(v)}$ is a Dirac measure $\zeta^{(v)}\left(x\right)=\delta\left(x-\frac{1}{2}\right)$.
\end{rem*}
~\\
As is implied by choice of notation, the distribution $\zeta^{(v)}$
indeed depends on the particular vertex $v\in\V$. We demonstrate
this in Figure \ref{fig: graph rho example},(iii) where we compare
between the probability distributions of two vertices of different
degrees from the same graph. In addition, Figure \ref{fig: graph rho example},(vi)
shows a comparison between the probability distributions of two vertices
of the same degree from different graphs. The numerics suggest that
the distributions are different, which implies that $\zeta^{(v)}$
may depend on the graph connectivity and not only on the degree of
the vertex. It is of interest to further investigate this distribution,
$\zeta^{(v)}$, and in particular its dependence on the graph's properties.

\begin{figure}
\begin{centering}
(i) $\quad$$\Gamma_{1}$\includegraphics[width=0.3\textwidth]{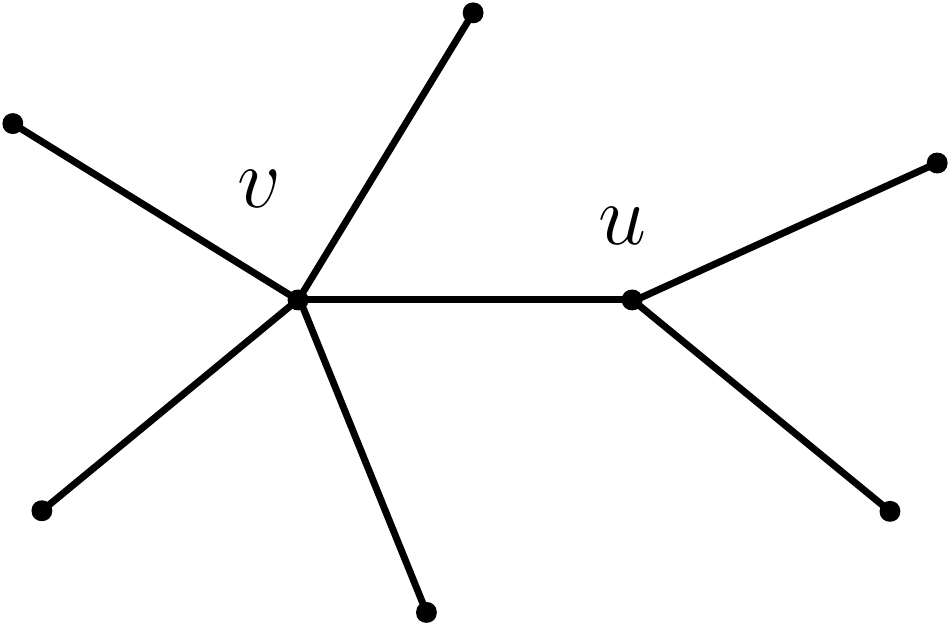}
\hfill{}(ii) $\quad$$\Gamma_{2}$\includegraphics[width=0.25\textwidth]{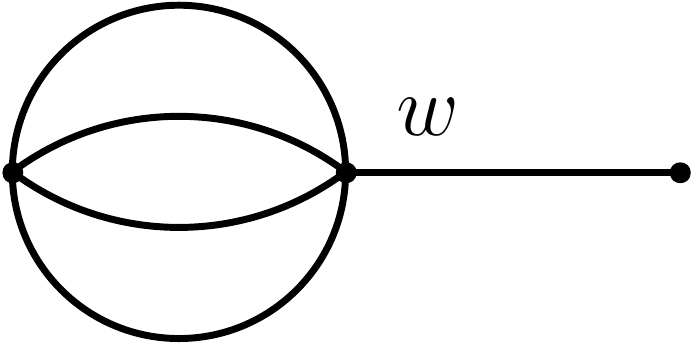}
\par\end{centering}
\begin{centering}
(iii)\includegraphics[width=0.7\textwidth]{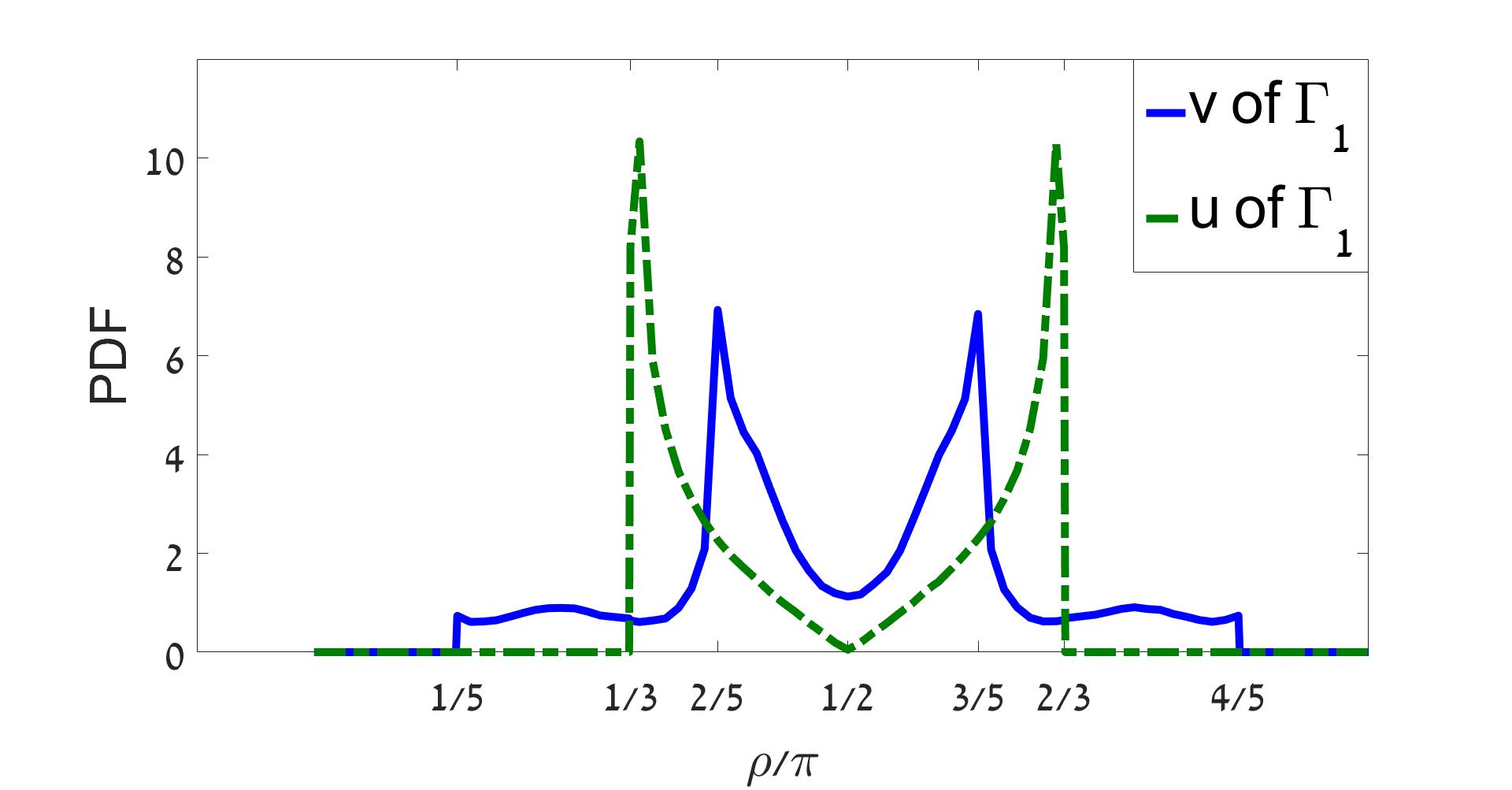}
\par\end{centering}
\begin{centering}
(iv)\includegraphics[width=0.7\textwidth]{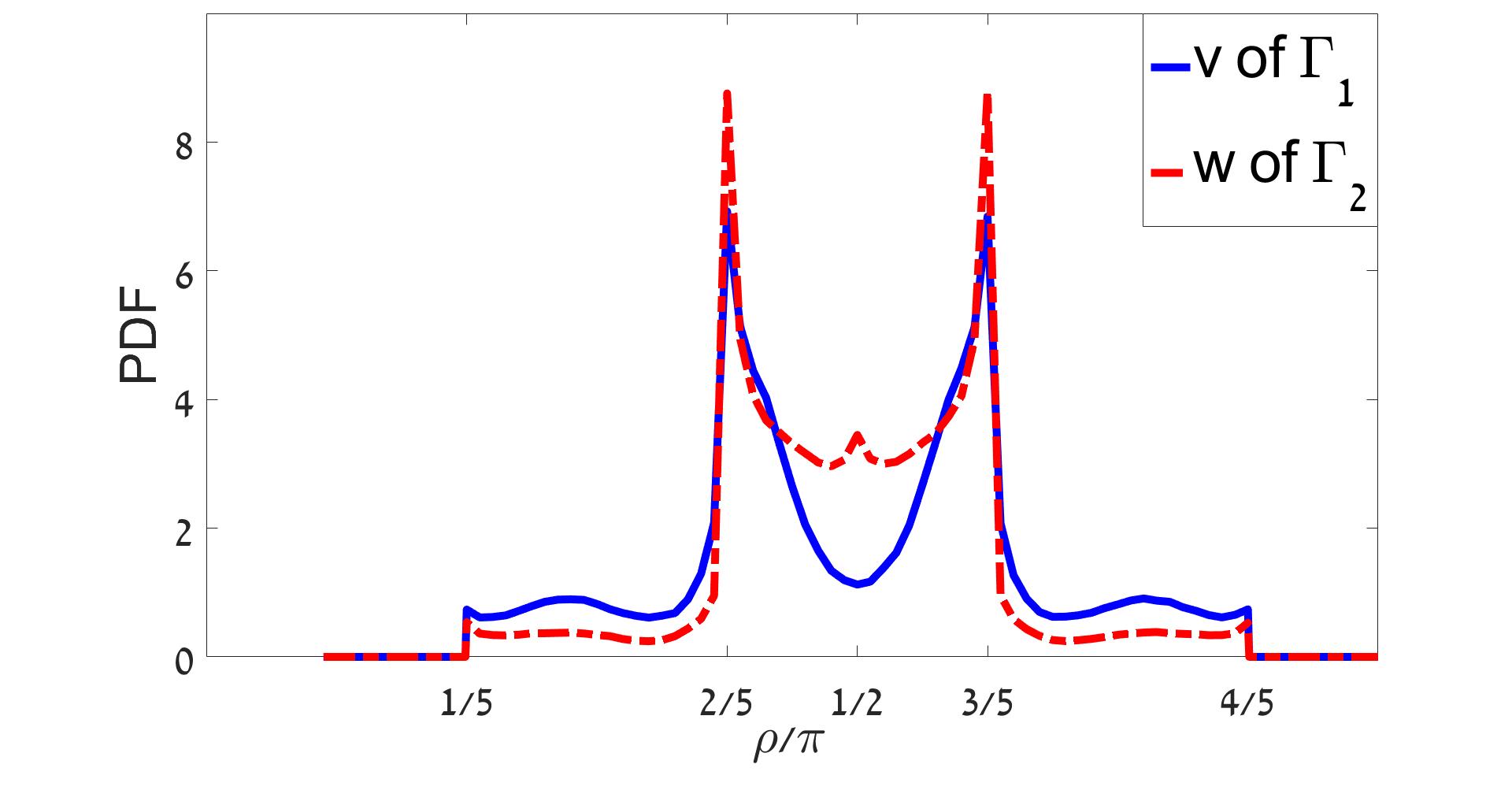}
\par\end{centering}
\centering{}\caption{(i) $\Gamma_{1}$, with vertices $v,~u$ of degrees $5,~3$, correspondingly.
(ii) $\Gamma_{2}$, with vertex $w$ of degree $5$. (iii) A probability
distribution function of $\frac{\rho}{\pi}$-values for the $\Gamma_{1}$
Neumann domains which contain $v$ (i.e., $\zeta^{(v)}$ in (\ref{eq:rho-distribution}))
compared with $\zeta^{(u)}$. (iv) Similarly, $\zeta^{(v)}$ compared
with $\zeta^{(w)}$.\protect \\
All the numerical data was calculated for the first $10^{6}$ eigenfunctions
and for a choice of rationally independent lengths.}
\label{fig: graph rho example}
\end{figure}

\section{Spectral position of $\Omega$}

By Lemma \ref{lem:Restriction-is-Neumann-eigenfunction-graph}, a
graph eigenvalue $k$ appears in the spectrum of each of its Neumann
domains. Exactly as in Definition \ref{def:Spectral-Position} for
manifolds, we define the spectral position of a Neumann domain $\Omega$,
as the position of $k$ in the spectrum of $\Omega$ and denote it
by $N_{\Omega}(k)$. Also, as in the manifold case, we have that $N_{\Omega}(k)\geq1$
for graphs and for exactly the same reason (see discussion after Definition
\ref{def:Spectral-Position}).

A useful tool in estimating the spectral position is the following
lemma, connecting the spectral position of $\Omega$ to the nodal
count of $\left.f\right|_{\Omega}$.
\begin{lem}
\cite{AloBan_Neumann}\label{lem:Spectral-pos-equals-nodal-count}
Let $\Gamma$ be a nontrivial standard graph, $f$ be a generic eigenfunction
of $\Gamma$ corresponding to an eigenvalue $k$ and let $\Omega$
be a Neumann domain of $f$, which is a tree graph. Then
\begin{enumerate}
\item \label{enu:lem-Spectral-pos-equals-nodal-count-1}$N_{\Omega}(k)=\phi(f|_{\Omega})$.
\item \label{enu:lem-Spectral-pos-equals-nodal-count-2}$N_{\Omega}(k)\leq\left|\d\Omega\right|-1$.
\\
In particular if $\Omega$ is a path graph then $N_{\Omega}(k)=1$.
\end{enumerate}
\end{lem}
The statement in (\ref{enu:lem-Spectral-pos-equals-nodal-count-1})
was proven in \cite{Berkolaiko07,Poketal_mat96,Schapotschnikow06}
under the assumption that $f|_{\Omega}$ is generic. This is indeed
the case since $f$ itself is generic and $\Omega$ is a tree graph
(see remark after Lemma \ref{lem:Restriction-is-Neumann-eigenfunction-graph}).
The statement in (\ref{enu:lem-Spectral-pos-equals-nodal-count-2})
follows as a combination of (\ref{enu:lem-Spectral-pos-equals-nodal-count-1})
with Proposition \ref{prop:Critical-and-nodal-points-graph},(\ref{enu:Critical-and-nodal-points-graph-3}).

We further remark on the applicability of the lemma above; it applies
for almost all Neumann domains. Indeed, for any given graph, all Neumann
domains except finitely many are star graphs or path graphs (by Lemma
\ref{lem:High-eigenvalue}), and those are particular cases of tree
graphs.\\

Next, we show that the value of the spectral position implies bounds
on the value of $\rho$, just as we had for manifolds (Proposition
\ref{prop:Rho-upper-bound-manifold}). For manifolds we got upper
bounds on $\rho$, whereas for graphs we get bounds from both sides.
\begin{prop}
\cite{AloBan_Neumann}\label{prop:Rho-bounds-spectral-position} Let
$\Gamma$ be a nontrivial standard graph, $f$ be an eigenfunction
of $\Gamma$ corresponding to an eigenvalue $k$ and let $\Omega$
be a Neumann domain of $f$. Then
\begin{equation}
\frac{\rho\left(\Omega\right)}{\pi}\geq\frac{1}{\left|\d\Omega\right|}\left(\frac{N_{\Omega}(k)+1}{2}\right).\label{eq:Rho-lower-bound}
\end{equation}
If $\Omega$ is a star graph then we further have the upper bound
\begin{equation}
\frac{\rho(\Omega)}{\pi}\leq\frac{1}{2}+\frac{1}{\left|\d\Omega\right|}\left(\frac{N_{\Omega}(k)-1}{2}\right).\label{eq:Rho-upper-bound}
\end{equation}
\end{prop}
\begin{rem*}
Note that if $N_{\Omega}(\lambda)>1$ then the bound in (\ref{eq:Rho-lower-bound})
improves the lower bound given in Proposition \ref{prop:rho-Bounds-graph}.
Similarly, if $\Omega$ is a star graph and $N_{\Omega}(\lambda)<\left|\d\Omega\right|-1$,
then the bound (\ref{eq:Rho-upper-bound}) improves the upper bound
given in Proposition \ref{prop:rho-Bounds-graph} for star graphs.
\end{rem*}
~

Next, we show that the spectral position has a well-defined probability
distribution. As in the previous section (Proposition \ref{eq:rho-distribution}),
we find that this distribution is best described when one focuses
on Neumann domains containing a particular graph vertex.
\begin{prop}
\cite{AloBan_Neumann}\label{prop:Prob-Dist-of-Spectral-Position}
Let $v\in\V$ of degree $d_{v}$. We have that the spectral position
probability,
\begin{equation}
P\left(N_{\Omega^{(v)}}=j\right):=\lim_{N\rightarrow\infty}\frac{\left|\left\{ n\le N\,:\,f_{n}\,\textrm{is generic and }N_{\Omega_{n}^{\left(v\right)}}(k_{n})=j\right\} \right|}{\left|\left\{ n\le N\,:\,f_{n}\,\textrm{is generic}\right\} \right|}\label{eq:Prob-Dist-of-Spectral-Position}
\end{equation}
is well defined. If we further assume that $d_{v}>2$ then
\begin{enumerate}
\item $P\left(N_{\Omega^{(v)}}=j\right)$ is supported in the set $j\in\left\{ 1,...,d_{v}-1\right\} $.
\item $P\left(N_{\Omega^{(v)}}=j\right)$ is symmetric around $\frac{d_{v}}{2}$,
i.e., $P\left(N_{\Omega^{(v)}}=j\right)=P\left(N_{\Omega^{(v)}}=d_{v}-j\right)$.
\end{enumerate}
If $d_{v}=1$ then $P\left(N_{\Omega^{(v)}}=j\right)=\delta_{j,1}$.
\end{prop}
By the proposition the support and the symmetry of the spectral position
probability depend on the degree of the vertex. Yet, vertices of the
same degree, but from different graphs may have different probability
distributions as is demonstrated in Figure \ref{fig: graph N dist},(iii).
In addition, we show in Figure \ref{fig: graph N dist},(vi) how the
conditional probability distribution of $\rho\left(\Omega\right)$
depends on the value of the spectral position $N_{\Omega}$ (compare
with the bounds (\ref{eq:Rho-lower-bound}),(\ref{eq:Rho-upper-bound})).

\begin{figure}[h]
\begin{centering}
(i) $\quad$$\Gamma_{1}$\includegraphics[width=0.3\textwidth]{s3s5-eps-converted-to.pdf}
\hfill{}(ii) $\quad$$\Gamma_{2}$\includegraphics[width=0.25\textwidth]{41man-eps-converted-to.pdf}
\par\end{centering}
\begin{centering}
(iii)\includegraphics[width=0.7\textwidth]{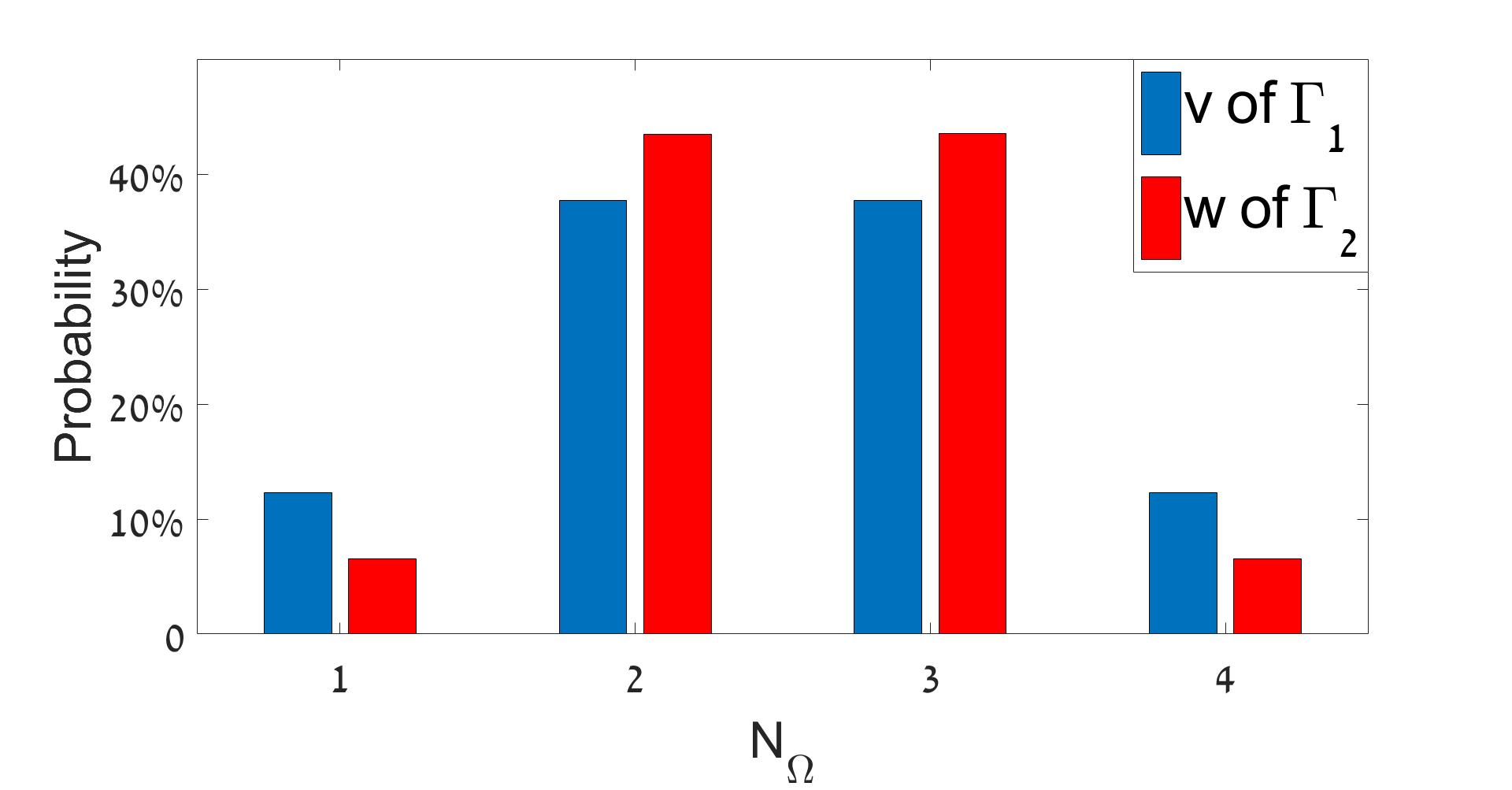}
\par\end{centering}
\begin{centering}
(iv)\includegraphics[width=0.7\textwidth]{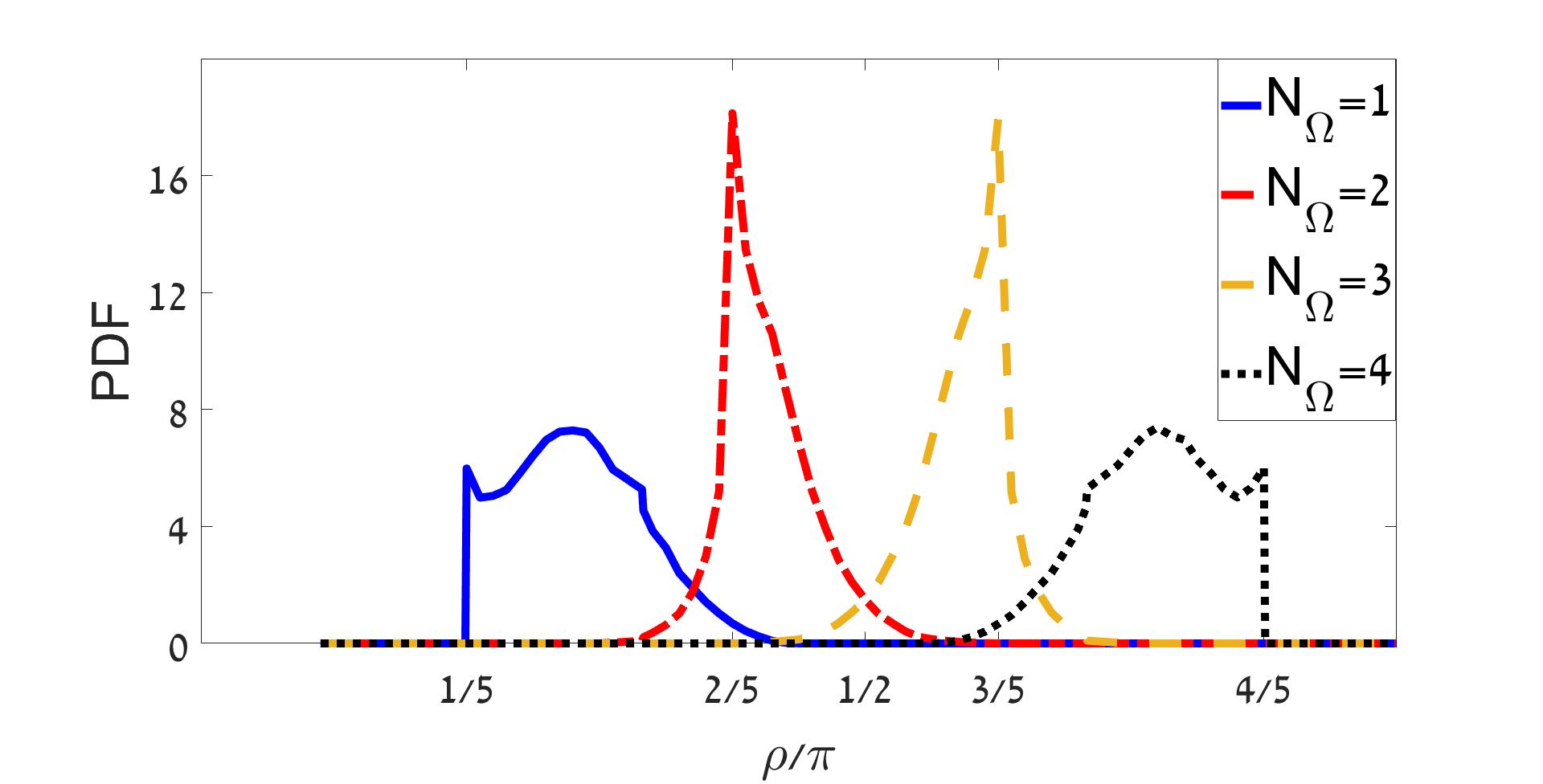}
\par\end{centering}
\centering{}\caption{(i) $\Gamma_{1}$, with vertex $v$ of degree $5$. (ii) $\Gamma_{2}$,
with vertex $w$ of degree $5$. (iii) The spectral position probability
$P\left(N_{\Omega^{(v)}}=j\right)$ for $v$ of $\Gamma_{1}$ compared
with $P\left(N_{\Omega^{(w)}}=j\right)$ for $w$ of $\Gamma_{2}$.
(iv) A probability distribution function of $\frac{\rho}{\pi}$-values
for the $\Gamma_{1}$ Neumann domains which contain $v$, conditioned
on the value of the spectral position $N_{\Omega_{n}^{\left(v\right)}}$.\protect \\
All the numerical data was calculated for the first $10^{6}$ eigenfunctions
for a choice of rationally independent lengths.}
\label{fig: graph N dist}
\end{figure}

\section{Neumann count}

In this section we present bounds on the number of Neumann points
and provide some properties of the probability distribution of this
number.
\begin{defn}
Let $\Gamma$ be a nontrivial standard graph and $\{f_{n}\}_{n=0}^{\infty}$
a complete set of its eigenfunctions. Denote by $\mu_{n}:=\mu(f_{n})$
and $\phi_{n}:=\phi(f_{n})$ the numbers of Neumann points and nodal
points respectively. We call the sequences $\{\mu_{n}\},\{\phi_{n}\}$
the Neumann count and nodal count, and the normalized quantities $\omega_{n}:=\mu_{n}-n,\,\,\,\sigma_{n}:=\phi_{n}-n$
are called the Neumann surplus and nodal surplus.
\end{defn}
\begin{prop}
\cite{AloBan_Neumann}\label{prop:Nodal-Neumann-bounds} Let $\Gamma$
be a nontrivial standard graph. Let $f_{n}$ be the $n^{\textrm{th}}$
eigenfunction of $\Gamma$ and assume it is generic. We have the following
bounds:
\begin{equation}
1-\beta\le\sigma_{n}-\omega_{n}\le\beta-1+\left|\d\Gamma\right|,\label{eq:Nodal-Neumann-bounds}
\end{equation}
and
\begin{equation}
1-\beta-\left|\d\Gamma\right|\leq\omega_{n}\leq2\beta-1,\label{eq:Neumann-Surplus-bounds}
\end{equation}
where $\beta=\left|\E\right|-\left|\V\right|+1$ is the first Betti
number of $\Gamma$.
\end{prop}
Moreover, both quantities $\sigma_{n}-\omega_{n}$ and $\omega_{n}$
have well defined probability distributions, as stated in what follows.

\begin{prop}
\cite{AloBan_Neumann} \label{prop:Nodal-Neumann-Distribution}~\\
\begin{enumerate}
\item \label{enu:Nodal-Neumann-Dist-1}The difference between the Neumann
and nodal surplus has a well defined probability distribution given
by
\begin{equation}
P\left(\sigma-\omega=j\right)=\lim_{N\rightarrow\infty}\frac{\left|\left\{ n\le N\,:\,f_{n}\,\textrm{is generic and }\sigma_{n}-\omega_{n}=j\right\} \right|}{\left|\left\{ n\le N\,:\,f_{n}\,\textrm{is generic}\right\} \right|}.\label{eq:prop-Nodal-Neumann-Distribution}
\end{equation}
Furthermore, it is symmetric around $\frac{1}{2}\left|\d\Gamma\right|$,\\
 i.e., $P\left(\sigma-\omega=j\right)=P\left(\sigma-\omega=\left|\d\Gamma\right|-j\right).$
\item \label{enu:Nodal-Neumann-Dist-2}The Neumann surplus has a well defined
probability distribution which is symmetric around $\frac{1}{2}\left(\beta-\left|\d\Gamma\right|\right).$
\end{enumerate}
\end{prop}
This proposition is in the spirit of the recently obtained result
for the distribution of the nodal surplus \cite{AloBanBer_cmp18}.
It was shown in \cite{AloBanBer_cmp18} that the nodal surplus, $\sigma$,
has a well defined probability distribution which is symmetric around
$\frac{1}{2}\beta$. The proof of Proposition \ref{prop:Nodal-Neumann-Distribution}
uses similar techniques to the proof of this latter result and appears
in \cite{AloBan_Neumann}.

The proposition above also has an interesting meaning in terms of
inverse problems. It is common to ask what one can deduce on a graph
out of its nodal count sequence, $\{\phi_{n}\}$ \cite{BanOreSmi_pspm08,BanShaSmi_jpa06,OreBan_jpa12}.
It was found in \cite{Ban_ptrsa14} that the nodal count distinguishes
tree graphs from others. The result already mentioned in \cite{AloBanBer_cmp18}
took a step further by showing that the nodal surplus distribution
reveals the graph's first Betti number, as twice the expected value
of the nodal surplus. However, it should be noted that all tree graphs
have the same nodal count, so that one cannot distinguish between
different trees in terms of the nodal count. Proposition \ref{prop:Nodal-Neumann-Distribution}
shows that the Neumann count, $\{\mu_{n}\}$ contains information
on the size of the graph's boundary, $\left|\partial\Gamma\right|$.
In particular, this enables the distinction between some tree graphs,
which was not possible before.

Different tree graphs with the same boundary size, $\left|\partial\Gamma\right|$,
have the same expected value for their Neumann count and are not distinguishable
in this sense. Nevertheless, we may wonder whether the boundary size
of a tree graph fully determines the probability distribution of its
Neumann count. We do not have an answer to this question yet and carry
on this exploration.

We end this section by noting that the bounds obtained in (\ref{eq:Neumann-Surplus-bounds})
on the Neumann surplus $\omega_{n}$ seem not to be strict, as we
observed in many examples. Furthermore, we conjecture the following
sharper bounds on $\omega_{n}$.
\begin{conjecture}
\label{conj:Neumann-surplus-bounds} The Neumann surplus is bounded
by
\[
-1-\left|\d\Gamma\right|\le\omega_{n}\le\beta+1.
\]
\end{conjecture}
Proving the bounds (\ref{eq:Neumann-Surplus-bounds}) on $\omega_{n}$
is done by combining the bounds on $\sigma_{n}-\omega_{n}$ (\ref{eq:Nodal-Neumann-bounds})
with the bounds $0\le\sigma_{n}\le\beta$ \cite{BerWey_ptrsa14}.
The bounds on both $\sigma_{n}-\omega_{n}$ and $\sigma_{n}$ are
known to be strict. Hence, if indeed the bounds on $\omega_{n}$ are
not strict, it implies that the nodal surplus, $\sigma_{n}$, and
the Neumann surplus, $\omega_{n}$, are correlated when considered
as random variables, which is an interesting result on its own.\vspace{4mm}

\part{Summary}

In this part we summarize the main results of this paper and focus
on the comparison between analogous statements on graphs and manifolds.
This is emphasized by using common terminology and notations for both
graphs and manifolds.

Let $f$ be an eigenfunction corresponding to the eigenvalue $\lambda$
and $\Omega$ be a Neumann domain of $f$. On manifolds, we have that
$\Omega$ and $\left.f\right|_{\Omega}$ are of a rather simple form;
$\Omega$ is simply connected; $\left.f\right|_{\Omega}$ has only
two nodal domains and its critical points are all located on $\partial\Omega$
(Theorem \ref{thm:topological-properties-manifolds}). On graphs,
the situation is similar, as almost all Neumann domains are either
star graphs or path graphs; it is possible to have other Neumann domains,
and even non simply connected ones, only if $\lambda$ is small enough
(Lemma \ref{lem:High-eigenvalue}). For graphs, $\left.f\right|_{\Omega}$
has two nodal domains if $\Omega$ is a path graph, but otherwise
may have more, with a global bound on this number (Proposition \ref{prop:Critical-and-nodal-points-graph},(\ref{enu:Critical-and-nodal-points-graph-3})).

The most basic property of Neumann domains is that $\left.f\right|_{\Omega}$
is a Neumann eigenfunction of $\Omega$ (Manifolds - Lemma \ref{lem:Restriction-is-Neumann-eigenfunction-manifold};
Graphs - Lemma \ref{lem:Restriction-is-Neumann-eigenfunction-graph}).
The eigenvalue of $\left.f\right|_{\Omega}$ is also $\lambda$ and
the interesting question is to find out what is the position of $\lambda$
in the spectrum of $\Omega$ - a quantity which we denote by $N_{\Omega}(\lambda)$
(Definition \ref{def:Spectral-Position}). The intuitive feeling at
the beginning of the Neumann domain study was that generically, $N_{\Omega}(\lambda)=1$
or that at least the spectral position gets low values.

The general problem of determining the spectral position is quite
hard for manifolds. The most general result we are able to provide
for manifolds (Proposition \ref{prop:Rho-upper-bound-manifold}) is
a lower bound given in terms of the geometric quantity $\rho$, which
is a normalized area to perimeter ratio (Definition \ref{def:Area-to-Perimeter-Manifolds}).
Interestingly, this result allows to estimate the spectral position
numerically; a numerical calculation of $\rho$ is rather easy compared
to the involved calculation of the spectrum of an arbitrary domain,
which is needed to determine a spectral position. This numerical method
allows to refute the belief that for manifolds, generically, $N_{\Omega}(\lambda)=1$.
For graphs, the quantity $\rho$ (Definition \ref{def:Area-to-perimeter-graphs}))
allows to bound the spectral position from both sides, for almost
all Neumann domains (Proposition \ref{prop:Rho-bounds-spectral-position}).
Two additional results we have for the spectral position on graphs
(but not for manifolds) are as follows. First, the spectral position
of $\Omega$ is given explicitly by the nodal count of $\left.f\right|_{\Omega}$,
and this yields an upper bound on the spectral position (Lemma \ref{lem:Spectral-pos-equals-nodal-count}).
Second, the spectral position has a limiting distribution which is
symmetric (Proposition \ref{prop:Prob-Dist-of-Spectral-Position}).
Another point of comparison is that an upper bound on the spectral
position, which we have for graphs, does not exist for manifolds.
We show by means of an example that the spectral position is unbounded
in the manifold case. This example is given in terms of separable
eigenfunctions on the torus. For this example, we show that although
the spectral position of half of the Neumann domains is unbounded,
it equals one for the other half (Theorem \ref{thm:Spectral-pos-star-and-lens}).
This finding might imply that even though $N_{\Omega}(\lambda)=1$
does not hold generically, there might be a substantial proportion
of Neumann domains, for which it does hold (see e.g., (\ref{eq:Assumption-on-Spectral-position})).
This is indeed the case for graphs where the spectral position of
each path graph Neumann domain equals one, and all of those form a
substantial proportion of all Neumann domains (their number as well
as their total length increase with the eigenvalue).

Finally, we discuss the Neumann domain count. On manifolds we count
the number of Neumann domains, while on graphs we count the number
of Neumann points. There is also a connection between the Neumann
count and the nodal count. On manifolds, we have that the difference
between the Neumann count and half the nodal count is non-negative
(Corollary \ref{cor:Neumann-count-bound-by-nodal-count}). On graphs,
the difference between the Neumann count and the nodal count is bounded
from both sides (Proposition \ref{prop:Nodal-Neumann-bounds}). As
for the Neumann count itself, it makes sense to consider it with a
normalization: $\frac{\mu(f_{n})}{n}$ on manifolds and $\mu(f_{n})-n$
on graphs. For graphs we provide general bounds on $\omega_{n}=\mu(f_{n})-n$
(Proposition \ref{prop:Nodal-Neumann-bounds}), but believe that those
are not sharp and conjecture sharper bounds (Conjecture \ref{conj:Neumann-surplus-bounds}).
The validity of the conjecture would also imply a correlation between
the nodal and the Neumann counts. In addition, $\omega_{n}$ possesses
a limiting probability distribution which is symmetric (Proposition
\ref{prop:Nodal-Neumann-Distribution}). The expected value of this
distribution stores information on the size of the graph's boundary,
$\left|\partial\Gamma\right|$; an information that is absent from
the nodal count. Which other graph properties may be revealed by this
distribution is still to be found. Turning back to manifolds, we treat
separable eigenfunctions on the torus and for those derive the probability
distribution of $\frac{\mu(f_{n})}{n}$ (Proposition \ref{prop:Distribution-Neumann-Count-manifold}).
This is to be viewed as the beginning of the analysis of Neumann count
on manifolds. Two approaches in which some progress can be made are
the following. One is getting a Courant-like bound of the form $\mu(f_{n})\leq h(n)$,
with $h$ being possibly a linear function. The second would be studying
the asymptotic behaviour, and for example showing that $\limsup_{n\rightarrow\infty}\mu(f_{n})=\infty$.
Both approaches are related to analogous results on nodal domains.
The first is tied to the Courant bound for nodal domains (whose strict
version does not hold for the Neumann count). The second is based
on a series of works on asymptotic growth of the nodal count \cite{GhoRezSar_gafa13,JunZel_jdg16,JunZel_mann16,Zelditch_jst16}
(see full description in Section \ref{subsec:mu_n-over-n}) together
with the basic bound (\ref{eq:Neumann-bounded-by-nodal-manifold})
which relates the nodal count and the Neumann count.

\section*{Acknowledgments}

We thank Alexander Taylor for providing the python code \cite{python_Taylor}
we used to generate the figures of Neumann domains on manifolds and
calculate the distribution of $\rho$. The authors were supported
by ISF (Grant No. 494/14).

\appendix

\section{\label{sec:Basic-Morse-Theory}Basic Morse Theory}

This section brings some basic statements in Morse theory which are
useful for understanding the first part of the paper. For a more thorough
exposition, we refer the reader to \cite{BanHur_MorseHomology04}.
Throughout the appendix we take $(M,g)$ be a compact smooth Riemannian
manifold of a finite dimension. At some points of the appendix we
specialize for the two-dimensional case and mention explicitly when
we do so.
\begin{defn}
Let $f:M\rightarrow\R$ be a smooth function.
\begin{enumerate}
\item $f$ is a Morse function if at every critical point, $\bs p\in\mathscr{C}(f)$,
the Hessian matrix, $\hess f|_{\boldsymbol{p}}$, is non-degenerate,
i.e., it does not have any zero eigenvalues.
\item The Morse index $\lambda_{\bs p}$ of a critical point $\bs p\in\mathscr{C}(f)$
is the number of negative eigenvalues of the Hessian matrix, $\hess f|_{\boldsymbol{p}}$.
\end{enumerate}
\end{defn}
The following three propositions may be found in \cite{BanHur_MorseHomology04}.
\begin{prop}
\cite[ Lemma 3.2 and Corollary 3.3]{BanHur_MorseHomology04} \\
If $f$ is a Morse function then the critical points of $f$ are isolated
and $f$ has only finitely many critical points.
\end{prop}
Next, we consider the gradient flow $\varphi:\mathbb{R}\times\,M\rightarrow M$
defined by (\ref{eq:flow}). For a particular $\bs x\in M$ we call
the image of $\varphi:\mathbb{R}\times\,\bs x\rightarrow M$, a gradient
flow line. Note that a gradient flow line, $\{\varphi(t;\thinspace\bs x)\}_{t=-\infty}^{\infty}$
has a natural direction dictated by the order of the $t$ values.
\begin{prop}
\cite[Propositions 3.18, 3.19]{BanHur_MorseHomology04}\\
\begin{enumerate}
\item Any smooth real-valued function $f$ decreases along its gradient
flow lines. The decrease is strict at noncritical points.
\item Every gradient flow line of a Morse function $f$ begins and ends
at a critical point. Namely, for all $\bs x\in M$ both limits $\lim_{t\rightarrow\pm\infty}\varphi(t,\thinspace\bs x)$
exist and they are both critical points of $f$.
\end{enumerate}
\end{prop}

\begin{prop}
[Stable/Unstable Manifold Theorem for a Morse Function]\cite[Theorem 4.2]{BanHur_MorseHomology04}\label{propstuns}\\
 Let $f$ be a Morse function and $\bs p\in\mathscr{C}(f)$. Then
the tangent space at $\bs p$ splits as
\[
T_{\bs p}M=T_{\bs p}^{s}M\thinspace\oplus\thinspace T_{\bs p}^{u}M,
\]
where the Hessian is positive definite on $T_{\bs p}^{s}M$ and negative
definite on $T_{\bs p}^{u}M$.

Moreover, the stable and unstable manifolds, (\ref{eq:Stable-Unstable-Def}),
are surjective images of smooth embeddings
\begin{align*}
T_{\bs p}^{s}M~ & \rightarrow~W^{s}(\bs p)~\subseteq~M\\
T_{\bs p}^{u}M~ & \rightarrow~W^{u}(\bs p)~\subseteq~M.
\end{align*}
Therefore, $W^{u}(\bs p)$ is a smoothly embedded open disk of dimension
$\lambda_{\bs p}$ and $W^{s}(\bs p)$ is a smoothly embedded open
disk of dimension $m-\lambda_{\bs p}$, where $m$ is the dimension
of $M$.
\end{prop}
Let us examine the implications of the results above in the particular
case of Morse functions on a two-dimensional manifold.
\begin{itemize}
\item If $\bs q$ is a maximum then $\lambda_{\bs q}=2$ and so $W^{u}(\bs q)$
is a two-dimensional open and simply connected set and $W^{s}(\bs q)=\{\bs q\}$.
\item If $\bs p$ is a minimum then $\lambda_{\bs p}=0$ and so $W^{s}(\bs p)$
is a two-dimensional open and simply connected set and $W^{u}(\bs p)=\{\bs p\}$.
\item If $\bs r$ is a saddle point then $\lambda_{\bs r}=1$ and so both
$W^{s}(\bs r)$ and $W^{u}(\bs r)$ are one-dimensional curves. Note
that $W^{s}(\bs r)\cap W^{u}(\bs r)=\{\bs r\}$ and so we get that
$W^{s}(\bs r)$ is a union of two gradient flow lines (actually even
Neumann lines) which end at $\bs r$. Similarly, $W^{u}(\bs r)$ is
a union of two gradient flow lines (Neumann lines) which start at
$\bs r$.
\end{itemize}
By Definition \ref{def:Neumann-Domains-and-Lines} we get that Neumann
domains are open two-dimensional sets and that the Neumann line set
is a union of one dimensional curves. Moreover, those sets are complementary.
Namely, the union of all Neumann domains together with the Neumann
line set gives the whole manifold \cite[Proposition 1.3]{BanFaj_ahp16}.\\

Next, we focus on a subset of the Morse functions, known as Morse-Smale
functions, described by the following two definitions.
\begin{defn}
We say that two sub-manifolds $M_{1},M_{2}\subset M$ intersect transversally
and write $M_{1}\pitchfork M_{2}$ if for every $\bs x\in M_{1}\cap M_{2}$
the tangent space of $M$ at $\bs x$ equals the sum of tangent spaces
of $M_{1}$ and $M_{2}$ at $\bs x$, i.e.
\begin{equation}
T_{\bs x}M=T_{\bs x}M_{1}+T_{\bs x}M_{2}.\label{eq:transversality-cond}
\end{equation}
This is also called the \emph{transversality condition}.
\end{defn}
\begin{defn}
A Morse function such that for all of its critical points $\bs{\bs p},\bs q\in\Cr$
the stable and unstable sub-manifolds intersect transversely, i.e.,
$W^{s}(\bs q)\pitchfork W^{u}(\bs p)$ is called a \emph{Morse-Smale}
function.
\end{defn}
Let us assume now that $M$ is a two-dimensional manifold and provide
a necessary and sufficient condition for a Morse function to be a
Morse-Smale function. First, for two critical points $\bs{\bs p},\bs q\in\Cr$,
the intersection $W^{s}(\bs p)\cap W^{u}(\bs q)$ may be non-empty
only for the following cases:
\begin{enumerate}
\item if $\bs p=\bs q$,
\item if $\bs p$ is a minimum and $\bs q$ is a maximum,
\item if $\bs p$ is a minimum and $\bs q$ is a saddle point,
\item if $\bs p$ is a saddle point and $\bs q$ is a maximum, or
\item if both $\bs p$ and $\bs q$ are saddle points.
\end{enumerate}
In the first four cases, it is straightforward to check that the transversality
condition is satisfied. In the last case we have that if $W^{s}(\bs p)\cap W^{u}(\bs q)\neq\emptyset$
then $W^{s}(\bs p)\cap W^{u}(\bs q)$ equals to the gradient flow
line (also Neumann line in this case) which asymptotically starts
at $\bs q$ and ends at $\bs p$. In such a case we get that for all
$\bs x\in W^{s}(\bs p)\cap W^{u}(\bs q)$, the tangent spaces obey
$T_{\bs x}W^{s}(\bs p)=T_{\bs x}W^{u}(\bs q)$ and those are one-dimensional,
so their sum cannot be equal to the two-dimensional $T_{\bs x}M$.
Therefore, in this case the transversality condition, (\ref{eq:transversality-cond})
is not satisfied and as a conclusion we get
\begin{prop}
\label{prop:Morse-Smale-on-2d}On a two-dimensional manifold, a Morse
function is Morse-Smale if and only if there is no Neumann line connecting
two saddle points.
\end{prop}
By the Kupka-Smale theorem (see \cite{BanHur_MorseHomology04}) Morse-Smale
gradient vector fields are generic among the set of all vector fields.
Currently, there is no similar genericity result regarding eigenfunctions
of elliptic operators which are Morse-Smale (in the spirit of \cite{Uhl_bams72,Uhl_ajm76}).
Our preliminary numerics suggest that Morse-Smale eigenfunctions are
indeed generic.\vspace{4mm}
\bibliographystyle{siam}
\bibliography{GlobalBib}

\end{document}